\documentclass{amsart}

\usepackage{color}
\usepackage{graphicx, amsmath, amssymb, amscd, amsthm, euscript, psfrag, amsfonts,bm}

\usepackage{hyperref}
\numberwithin{equation}{section}

\usepackage{amsmath,amsthm,indentfirst}
\usepackage{amssymb}
\usepackage{amsfonts}
\usepackage{amscd}
\usepackage[latin1]{inputenc}
\DeclareMathAlphabet{\mathpzc}{OT1}{pzc}{m}{it}
\theoremstyle{plain}
\newtheorem*{maintheorem*}{Main Theorem}
\newtheorem*{thm*}{Theorem}
\newtheorem*{remark*}{Remark}
\newtheorem*{conjecture*}{Conjecture}
\newtheorem*{prop*}{Proposition}
\newtheorem*{lem*}{Basic Lemma}
\newtheorem{thm}{Theorem}[section]
\newtheorem{cor}[thm]{Corollary}
\newtheorem{lem}[thm]{Lemma}
\newtheorem{prop}[thm]{Proposition}

\newtheorem*{sublemma*}{Sublemma}

\theoremstyle{definition}

\newtheorem*{proofc*}{Proof of Theorem C}

\newtheorem{definition}[thm]{Definition}

\newtheorem{remark}[thm]{Remark}

\def\bbz{\mathbb{Z}}

\def\bbf{\mathbb{F}}
\def\bbr{\mathbb{R}}
\def\bba{\mathbb{A}}
\def\bbb{\mathbb{B}}

\def\bbe{\mathbb{E}}
\def\bbh{\mathbb{H}}
\def\bbn{\mathbb{N}}
\def\bbg{\mathbb{G}}

\def\bbv{\mathbb{V}}
\def\bbu{\mathbb{U}}
\def\bbp{\mathbb{P}}
\def\bbl{\mathbb{L}}
\def\bbs{\mathbb{S}}
\def\bbw{\mathbb{W}}
\def\bbm{\mathbb{M}}

\def\fcal{\mathcal{F}}
\def\ucal{\mathcal{U}}
\def\vcal{\mathcal{V}}
\def\ocal{\mathcal{O}}
\def\scal{\mathcal{S}}
\def\acal{\mathcal{A}}
\def\bcal{\mathcal{B}}

\def\mcal{\mathcal{M}}
\def\tcal{\mathcal{T}}
\def\rcal{\mathcal{R}}

\def\Bfrak{\mathfrak{B}}

\def\Sfrak{\mathfrak{S}}
\def\ybf{\mathbf{y}}
\def\xbf{\mathbf{x}}

\def\qpz{{\mathpzc{q}}}
\def\ppz{{\mathpzc{p}}}

\DeclareMathOperator\Spec{Spec}

\def\SL{{\rm{SL}}}

\def\Ad{{\rm{Ad}}}
\def\GL{{\rm{GL}}}
\def\Lie{{\rm Lie}}
\def\h{\hspace{1mm}}
\def\hh{\hspace{.5mm}}
\def\la{\lambda}

\def\kt{k_{\mathcal{T}}}
\def\lt{l_{\mathcal T}}

\def\vare{\varepsilon}

\def\tprod{\prod_{{v}\in\mathcal{T}}}
\def\wpg{W_G^+(s)}
\def\wmg{W_G^-(s)}
\def\wpgp{W_{G'}^+(s)}
\def\wmgp{W_{G'}^-(s)}
\def\wph{W_H^+(s)}
\def\wmh{W_H^-(s)}
\def\wpmg{W_G^\pm(s)}

\def\zg0{Z_{G_{w}}(s)}

\def\zg{Z_G(s)}
\def\zgp{Z_{G'}( s)}

\def\tphi{\widetilde{\phi}}
\def\rest{\mathcal{R}_{k/k^q}}

\def\be{\begin{equation}}
\def\ee{\end{equation}}

\def\field{{k}}
\def\sfield{{l}}
\def\gfield{{K}}
\def\vintegers{{\mathcal O_\tcal}}
\def\oldaleph{{\mathsf C}}

\begin{document}

\title[Measure rigidity for solvable group action]{Characteristic free measure rigidity for the action of solvable groups on homogeneous spaces}

\author{Amir Mohammadi}
\address{Mathematics Dept, University of California, San Diego, CA 92093-0112}
\email{ammohammadi@ucsd.edu}
\thanks{A.M.\ was supported in part by NSF Grants \#1200388, \#1500677 and \#1128155, and Alfred P.~Sloan Research Fellowship.}

\author{Alireza Salehi Golsefidy}
\address{Mathematics Dept, University of California, San Diego, CA 92093-0112}
\email{golsefidy@ucsd.edu}
\thanks{A.S-G.\ was supported in part by NSF Grants  \#1001598, \#1160472, \#1303121 and \#1602137, and by Alfred P.~Sloan Research Fellowship.}


\begin{abstract}
We classify measures on a homogeneous space 
which are invariant under a certain solvable 
subgroup and ergodic under its unipotent
radical. Our treatment is independent of characteristic.  
As a result we get the first measure classification for the action of semisimple subgroups 
without any characteristic restriction.  
\end{abstract}

\maketitle

\section{introduction}\label{sec;intro} 
Ratner's celebrated work,~\cite{Rat3,Rat5,Rat6} see also~\cite{MT}, 
classifies all probability measures invariant and ergodic under a one parameter unipotent subgroup
in the setting of homogeneous spaces over local fields of characteristic zero.    
In positive characteristic setting, however, classification theorems in this generality are not yet available.

Roughly speaking, the main technical difficulty arises from the fact that 
the image of a polynomial map over a field
of positive characteristic can lie in a proper subfield, and hence the image may be small. 
This simple fact enters our study as follows.
The divergence of the orbits of two nearby points under a unipotent 
group is governed by a certain polynomial like map, see~\S\ref{sec;quasi}.
Now slow growth of polynomial maps and Birkhoff ergodic theorem imply that $\mu$ is
invariant under the image of a polynomial like map, see~\S\ref{sec:basic-lemma}.
This way, we construct a higher dimensional subgroup which leaves $\mu$ invariant.   
In the positive characteristic setting, however, this construction only guarantees the dimension is increased by a (possibly non-constant) fraction at each step. Thus, there is no a priori reason for this process to terminate.

In recent years there have been some partial results in this direction, e.g.~\cite{EG, AmirHoro,EM}. 
In~\cite{EG} measure rigidity for the action of semisimple groups was proved under the assumption that the characteristic is {\em large}. 
Our account here, in particular, removes the characteristic restriction from the main result in~\cite{EG}.
It is plausible that for {\em large characteristics}, 
one can carry out the proof in~\cite{MT} and get a measure rigidity for the action of unipotent groups.

Removing the characteristic restriction introduces serious technical difficulties. 
Resolving these difficulties requires not only a comprehensive use of the main ideas and techniques from~\cite{MT},
but also an extensive use of the theory of algebraic groups.
In contrast, the proof in~\cite{EG} relies on a simpler argument which goes back to~\cite{E}. 

In order to properly state our main results we need to make use of Weil's restriction of scalars for possibly non separable extensions. But before we set up the general formulation, we start with an example which is essentially as hard as the general case. Consider $G=\SL_d\Bigl(\bbf_q((t))\Bigr)$ and let $\Gamma\subset G$ be a discrete subgroup. 
Let $B=SU$ be the group of upper triangular matrices in a {\em copy} of $\SL_2\Bigl(\bbf_q((t^{q}))\Bigr)$ in $G$. Our result classifies measures on $G/\Gamma$ which are invariant under $B$ and ergodic for $U.$ 
Note, however, to get a copy of $\SL_2\Bigl(\bbf_q((t^{q}))\Bigr)$ in $G$ it is inevitable to view $G$ as the set of $\bbf_q((t^{q}))$-points of 
\[
\rcal_{\bbf_q((t))/\bbf_q((t^{q}))}(\SL_d)
\]
where $\rcal$ denotes Weil's restriction of scalars, see Definition~\ref{weil-rest}. 
This is an indication of the subtlety involved in the setup.

Let $\mathcal{T}$ be a finite set, and 
for any ${v}\in\mathcal{T}$ let $k_{v}$ be a local field; 
set $\kt=\tprod k_{v}$.  
For any ${v}\in\mathcal{T}$ let $\bbg_{v}$ be a $k_{v}$-algebraic group and let $G_{v}=\bbg_{v}(k_{v}).$ 
Define 
\[
\mbox{$\bbg=\tprod\bbg_{v}$ and $G=\tprod G_v,$}
\] 
and let $\Gamma$ be a discrete subgroup of $G.$


Fix an element ${w}\in\mathcal{T}$ once and for all and let 
$k'$ be a closed subfield of $k_{w}.$ Suppose $k''/k'$
is a finite extension of $k',$ note that $k''$ is not necessarily assumed to be a subfield of $k_w.$
Let $\bbh'$ be a connected $k''$-almost simple, $k''$-group which is isotropic over $k'',$ 
and put 
\[
 \bbh=\rcal_{k''/k'}(\bbh').
\] 
In particular, we have $\bbh(k')=\bbh'(k'').$

Fix a non central cocharacter $\lambda$ of $\bbh,$ that is $\lambda:\bbg_m\to\bbh$ is a non central homomorphism defined over $k';$ such homomorphism exists thanks to the fact that $\bbh'$ is $k''$-isotropic,~\cite[App.~C]{CGP}.
Put $\bbs=\lambda(\bbg_m).$ 
Let $s'\in\bbs(k')$ be an element which generates an unbounded subgroup
and set $\bbu=\bbw^+_{\bbh}(s'),$ see~\S\ref{sec;notation-algebra} for the notation.

For the base change $\bbh\times_{k'}k_w$, we fix a $k_w$-homomorphism 
$\iota:\bbh\times_{k'}k_w\to\bbg_w$ with a solvable kernel, and put 
\[
\mbox{$H=\iota(\bbh(k')),$ $S=\iota(\bbs(k')),$ and $U=\iota(\bbu(k')).$}
\]

We recall the following definition. Let $M$ be a locally compact second countable group and 
suppose $\Lambda$ is a discrete subgroup of $M$. 
Let $\mu$ be a probability measure on $M/\Lambda$ and let $\Sigma=\{g\in M: g_*\mu=\mu\}$.
We call $\mu$ {\it homogeneous} 
if $\mu$ is the $\Sigma$-invariant probability measure on a closed orbit $\Sigma x$ for some $x\in M/\Lambda.$

\begin{thm}[Solvable case]\label{t;U-measure-class}\label{thm:main}
Let $\mu$ be a probability measure on $G/\Gamma$ which is invariant under the action of $SU$ and is $U$-ergodic. Then $\mu$ is a homogeneous measure.
\end{thm}

The following presents an interesting special case which is a direct corollary of Theorem~\ref{t;U-measure-class}
based on the generalized Mautner phenomenon.
\begin{thm}[Semisimple case]\label{c;H-measureclass}
Let 
\begin{itemize}
\item $k'\subset k_w$ be a closed subfield, 
\item let $\bbh_0$ be a connected, simply connected, $k'$-almost simple, $k'$-isotropic semisimple group, 
\item let $\iota:\bbh_0\times_{k'}k_w\to\bbg_w$ be a $k_w$-homomorphism with a finite central kernel.
\end{itemize} 
Set $H_0=\iota(\bbh_0(k')).$
If $\mu$ is an $H_0$-invariant ergodic measure on $G/\Gamma,$ then $\mu$ is a homogeneous measure.
\end{thm} 

\begin{proof}
Since $\bbh_0$ is $k'$-isotropic, by~\cite[Ch. I, Prop.\ 1.6.3]{Mar5} we have that there exists a $k'$-homomorphism 
\[
j:\SL_2\to\bbh_0
\]
with a finite central kernel. 
Let $L=\iota\Bigl(j(\SL_2(k'))\Bigr)$ and let $B=SU$ be the image under $j\circ\iota$ of the group of
upper triangular matrices in $\SL_2(k').$

Since $\bbh_0$ is simply connected, $k'$-simple, and $k'$-isotropic, it follows 
from~\cite[Ch.I, Thm.\ 2.3.1 and Ch. II, Lemma 3.3]{Mar5} that $\mu$ is $S$-ergodic.
In particular, $\mu$ is $L$-ergodic. 
Now by~\cite[Ch. II, Lemma 3.4]{Mar5} we get that $\mu$ is $U$-ergodic.
Therefore, the claim follows from Theorem~\ref{t;U-measure-class}. 
\end{proof}

For a homogeneous measure $\mu$, the group $\Sigma=\{g: g_*\mu=\mu\}$ has an algebraic description. 
The statement of this refinement involves definitions and notation which will be developed later, see Theorem~\ref{algeb-measure}. When $\Gamma$ is an arithmetic lattice, we also give a connection between the arithmetic structure of $\Gamma$ and the algebraic description of $\Sigma$, see Theorem~\ref{thm:arithmetic-irred} and Theorem~\ref{thm:arithmetic-red}.  
This connection and Theorem~\ref{c;H-measureclass} are crucial ingredients 
in~\cite{ELM-Torus} where a measure rigidity for the action of diagonalizable groups is proved.

Now we give a brief description of our approach and highlight the main difficulties. 
We construct extra invariance for the measure using quasi-regular maps and utilize entropy, similar to the strategy in~\cite{MT}.
As we described above, however, several problems arise along the way.
The source of the main technical difficulties is the fact that topologically closed unipotent subgroups in positive characteristic are far from being algebraic. 
This issue appears in our proof, as we need to control the group generated by the image of a polynomial map.

In our setting, since $\mu$ is invariant under $S$, we need to understand the structure of 
topologically closed unipotent subgroups which are normalized by $S$.
In \S\ref{sec;algeb-lemma}, we show that such unipotent groups have an algebraic structure which is controlled by the weights of the conjugation action of $S$.
Our argument is based on a recent treatment, \cite[App.\ B]{CGP}, of the fundamental work of Tits on unipotent groups in positive characteristic.

As was mentioned above, \S\ref{sec:arithmetic} is devoted to providing 
a connection between the arithmetic structure of $\Gamma$ and the algebraic description of $\Sigma$. 
Our argument is based on the structure theory of {\em pseudo reductive} groups developed in~\cite{CGP}, to get a local model, and a descent argument from~\cite{Pink}, to get a global model.

Let us briefly recall the definition of a  pseudo reductive group and how it appears in our work.
An $l$-group with no nontrivial, normal, unipotent $l$-subgroup is called {\it pseudo reductive}.
In the setting of homogeneous dynamics we have to work with topologically closed subgroups of $G$.
As we mentioned above these groups are not necessarily algebraic, but under some conditions, they become 
algebraic after using Weil's restriction of scalars and viewing $G$ as an algebraic group over a smaller field.
So groups of the form $\rcal_{k/l}(\bbl)$ where $k/l$ is a finite extension naturally appear in our work.
When $\bbl$ is a connected, reductive, $l$-group and $k/l$ is inseparable, 
$\rcal_{k/l}(\bbl)$ is pseudo reductive but {\em not} reductive.

It is worth mentioning that, in our work, 
the structure theory of pseudo reductive groups is used only to reveal a more precise local description of $\Sigma$, however, our measure classification result, Theorem~\ref{thm:main}, {does not} rely on this structure theory.

We close the introduction with the following remark. 
It is desirable to classify all $SU$-orbit closures. Since $SU$ is amenable, 
such a topological rigidity would follow from a combination of linearization techniques and a classification of $SU$-{invariant, ergodic probability measures}. 
Here, however, we have made and used, in a crucial way, the stronger assumption that $\mu$ is $SU$-invariant and $U$-{\em ergodic.}  
Indeed if we assume $\mu$ is $SU$-invariant and only $SU$-ergodic, then $\mu$ is not necessarily $U$-ergodic, 
even though, by Mautner phenomenon, it is $S$-ergodic. In particular, $S$ acts ergodically on the space of $U$-ergodic components of $\mu$ and the invariance group of a $U$-ergodic component of 
$\mu$ is not necessarily normalized by $S$.
So we can not get a good algebraic description of this invariance group.

{\bf Acknowledgement:} We would like to thank 
M.~Einsiedler, A.~Eskin, M.~Larsen, E.~Lindenstrauss, G.~Margulis, and G.~Prasad for their interest 
in this project and several insightful conversations.


\section{Notation and preliminary lemmas}\label{sec;notation-algebra}

Let $\mathcal{T}$ be a finite set and let $k_{v}$ be a local field for all ${v}\in\mathcal{T};$ 
define $\kt=\tprod k_{v}$ as in the introduction. 
We endow $\kt$ with the norm $|\cdot|=\max_{{v}\in S} |\cdot|_{{v}}$ 
where $|\cdot|_{{v}}$ is a norm on $k_{{v}}$ for each ${v}\in \mathcal{T}.$ 

\subsection{$\kt$-algebraic groups.}\label{sec;alg-group} 
A $\kt$-algebraic group $\mathbb{M}$ (resp.\ $\kt$-algebraic variety $\mathbb{M}$) is the formal product of 
$\tprod \mathbb{M}_{{v}}$ of $k_{{v}}$-algebraic groups 
(resp.\ $\tprod \mathbb{M}_{{v}}$ of $k_{{v}}$-algebraic varieties). 
The usual notions from elementary algebraic geometry 
e.g.\ regular maps, rational maps, rational point etc.\ are defined fiberwise. 
We will use these notions without further remarks.  

There are two topologies on $\mathbb{M}(\kt),$ the Hausdorff topology and the Zariski topology.
We will make it clear when referring to the Zariski topology. 
Hence, if in a topological statement we do not give reference 
to the particular topology used, then the one which is being considered is the Hausdorff topology.

Let $\bbm$ be a $\kt$-group.
An element $e\neq g\in\mathbb{M}(\kt)$ is an element of class $\mathcal{A}$ if $g=(g_{{v}})_{{v}\in \mathcal{T}}$ is diagonalizable over $\kt,$ and for all ${v}\in \mathcal{T}$ 
the component $g_{{v}}$ has eigenvalues which are integer powers of the uniformizer 
$\varpi_{{v}}$ of $k_{{v}}$.

Given a subset $B\subset \mathbb{M}(\kt)$ we let $\langle B\rangle$ denote the closed (in the Hausdorff topology) group generated by $B.$

\subsection{Pseudo-parabolic subgroups}\label{sec:pseudo-parabolic}
Let $k$ be a local field of positive characteristic.
Suppose $\mathbb{M}$ is a connected $k$-group, 
and let $\lambda:\bbg_m\rightarrow{\bbm}$ be a non central homomorphism defined over $k.$ 
Define $-\lambda(a)=\lambda(a)^{-1}$ for all $a\in k^\times.$ 
As in~\cite[\S13.4]{Sp} and~\cite[Ch.~2 and App.~C]{CGP}, we let $\bbp_{\bbm}(\lambda)$ 
denote the closed subgroup of ${\bbm}$ formed by those elements $x\in{\bbm}$ 
such that the map $\lambda(a)x\lambda(a)^{-1}$ extends to a map from $\bbg_a$ into ${\bbm}.$ 

Let $\bbw_\bbm^+(\lambda)$ be the normal subgroup of $\bbp_{\bbm}(\lambda),$ 
formed by $x\in\bbp(\lambda)$ such that $\lambda(a)x\lambda(a)^{-1}\rightarrow e$ as $a$ goes to $0.$ 
The centralizer of the image of $\lambda$ is denoted by $\bbz_{\bbm}(\lambda).$ 
Similarly define $\bbw_{\bbm}^+(-\lambda)$ which we will denote by $\bbw_{\bbm}^-(\lambda).$ 

The multiplicative group $\bbg_m$ acts on $\Lie({\bbm})$ via $\lambda,$ and the weights are integers. 
The Lie algebras of $\bbz_{\bbm}(\lambda)$ and $\bbw_{\bbm}^\pm(\lambda)$ 
may be identified with the weight subspaces of this action corresponding to the zero, positive and negative weights. 
It is shown in~\cite[Ch.~2 and App.~C]{CGP}, see also~\cite[\S13.4]{Sp} and~\cite{BT2}, 
that $\bbp_{\bbm}(\lambda),$ $\bbz_{\bbm}(\lambda)$ and $\bbw_{\bbm}^{\pm}(\lambda)$ 
are $k$-subgroups of ${\bbm}.$ Moreover, $\bbw_{\bbm}^\pm(\lambda)$ is a normal subgroup of $\bbp_{\bbm}(\lambda)$ and the product map 
\[
\mbox{$\bbz_{\bbm}(\lambda)\times \bbw_{\bbm}^+(\lambda)\rightarrow\bbp_{\bbm}(\lambda)$ 
is a $k$-isomorphism of varieties.} 
\]

A pseudo-parabolic $k$-subgroup of $\bbm$ is a group of the form
$\bbp_\bbm(\la)R_{u,k}(\bbm)$ for some $\la$ as above 
where $R_{u,k}(\bbm)$ denotes the maximal,
connected, normal unipotent $k$-subgroup of $\bbm,$~\cite[Def.~2.2.1]{CGP}.

We also recall from~\cite[Prop.~2.1.8(3)]{CGP} that the product map
\begin{equation}\label{opposite-horo}
\bbw^-_{\bbm}(\lambda)\times\bbz_{\bbm}(\lambda)\times\bbw_{\bbm}^+(\lambda)\rightarrow{\bbm}\;\text{ is an open immersion of $k$-schemes.}
\end{equation} 
It is worth mentioning that these results are generalization to arbitrary groups of analogous and 
well known statements for reductive groups. 

Let $M={\bbm}(k),$ and put 
\[
\mbox{$W_M^\pm(\la)=\bbw^\pm_{\bbm}(\lambda)(k),$ and $Z_M(\la)=\bbz_{\bbm}(\lambda)(k).$ }
\]
From~\eqref{opposite-horo} we conclude that $W_M^-(\la)Z_M(\la)W_M^+(\la)$ 
is a Zariski open dense subset of $M,$ which contains a neighborhood of 
identity with respect to the Hausdorff topology.

For any $\la$ as above define 
\be\label{eq:M-+-la}
M^+(\la):=\langle W_M^+(\la),W_M^-(\la)\rangle.
\ee

\begin{lem}\label{lem:w-pm-normal}
The group $M^+(\la)$ is a normal subgroup of $M$ for any $\la$ as above.
Moreover, $M^+(\la)$ is unimodular.
\end{lem}

\begin{proof}
Since $D:=W_M^-(\la)Z_M(\la)W_M^+(\la)$ is a Zariski open dense subset of $M,$
we have $DD=M.$ In particular, $\langle W_M^-(\la),Z_M(\la),W_M^-(\la)\rangle=M.$
This together with the fact that $W_M^{\pm}(\la)$ is normalized by $Z_M(\la)$ implies the first claim.

We now show that $M^+(\la)$ is unimodular. The modular function 
is a (continuous) homomorphism from $M^+(\la)$ into the multiplicative group $\bbr^+.$
However, $M^+(\la)$ is generated by unipotent subgroups and since 
${\rm char}(k)$ is positive, unipotent subgroups are torsion. Hence, 
the modular function is trivial; the claim follows.     
\end{proof}





Given an element $s\in M$ from class $\mathcal A$. There is $\la:\bbg_m\to\bbm$
so that $s=\la(a)$ for some $a\in k$ with $|a|>1.$ Then we define
\be\label{eq:def-W-pm-s}
W^\pm_M(s):=W^\pm(\la).
\ee


\subsection{}\label{sec:k-closed}
When working with algebraic groups over a non perfect field $k,$ say with characteristic $p>0$, 
it is convenient to use the language of group schemes. 
With the exceptions of \S\ref{sec:psd-red} and \S\ref{sec:arithmetic} which are independent of 
the rest of the paper, we have tried to avoid this language.  
However, one should note that certain natural objects, e.g.\ kernel of a $k$-morphism, 
are not necessarily defined over the base field as linear algebraic groups in the sense of~\cite{B1} or~\cite{Sp}. They are so called $k$-closed\footnote{Let us remark that over a perfect field the notation of $k$-closed and that of
a variety defined over $k$ coincide.}; a notion which we now define.  

\begin{definition}[\cite{B1}, AG, \S12.1]\label{def;alg-mfld}
Let $\Omega$ be an algebraically closed field which contains $k$ 
and let ${\bbm}=\Spec(\Omega[x_1,\cdots,x_n]/I)$ be a variety.
The variety $\bbm$ is called $\field$-closed if $I={\rm rad}\Bigl({\Omega[x_1,\cdots,x_n]J}\Bigr)$ where $J$ is an ideal in $k[x_1,\cdots,x_n].$ 
\end{definition}

A subset of $k^n$ will be called $k$-closed 
if it may be realized as the $k$-points of a $k$-closed subset of $\bbg_a^n$.

If $M\subset k^n$ is a set which is the zero set of an ideal 
$J$ in $k[x_1,\cdots,x_n],$ then $M$ is a $k$-closed set; 
this is how the $k$-closed sets arise in our study. 
We also note that if ${\bbm}$ is $k^{p^{-\infty}}$-closed, then it is also $k$-closed.

If we start with a subset of the $k$-points 
of a variety and take the Zariski closure of this set, then
we get a variety defined over $k,$ see~\cite[Lemma 11.2.4(ii)]{Sp}. 
The next lemma is a more general formulation.

\begin{lem}[Cf.~\cite{CGP}, Lemma C.4.1]\label{l;group-scheme}
Let ${\bbm}$ be a scheme locally of finite type over a field $k$. 
There exists a unique geometrically reduced, closed subscheme ${\bbm}'\subset{\bbm}$ such that 
${\bbm}'(k')={\bbm}(k')$ for all separable field extensions $k'/k.$ 
The formation of ${\bbm}'$ is functorial in ${\bbm},$ and commutes with the formation of 
products over $k$ and separable extensions of the ground field. 
In particular, if ${\bbm}$ is a $k$-group scheme (not necessarily smooth), 
then ${\bbm}'$ is a smooth $k$-subgroup scheme.    
\end{lem}

Let us also recall the definition of the Weil's restriction of scalars.
\begin{definition}\label{weil-rest}
Let $k$ be a field and $k'$ a subfield of $k$ such that $k/k'$ is a finite extension, and let $\bbm$ be an affine $k$-variety. 
The Weil's restriction of scalars $\rcal_{k/k'}(\bbm)$ is the affine $k'$-scheme satisfying the following universal property 
\begin{equation}\label{e-weil-rest}
\rcal_{k/k'}(\bbm)(B)=\bbm(k\otimes_{k'}B)
\end{equation}
for any $k'$-algebra $B.$   
\end{definition}
 
\subsection{Ergodic measures on algebraic varieties.}\label{sec;alg-measure}
Let $\mathbb{M}$ be a $\kt$-group and let $M={\bbm}(\kt)$. 
Suppose $B\subset M$ is a group which is generated by one parameter $\kt$-split 
unipotent algebraic subgroups and by one dimensional $\kt$-split tori.
Let $\Lambda$ be a discrete subgroup of $M$ and put $\pi:M\rightarrow M/\Lambda$ to be the natural projection. 

\begin{lem}[Cf.~\cite{MT}, Proposition 3.2]\label{zd-measure1}
Let $\mu$ be a $B$-invariant and ergodic Borel probability measure on $M/\Lambda$. 
Suppose $\mathbb{D}$ is a $\kt$-closed subset of ${\bbm}$ and put $D=\mathbb{D}(\kt).$ 
If $\mu(\pi(D))>0,$ then there exists a connected $\kt$-algebraic subgroup $\mathbb{E}$ of ${\bbm}$ 
such that $B\subset E:=\mathbb{E}(\kt),$ and a point $g\in D$ such that $Eg\subset D$ and $\mu(\pi(Eg))=1.$
Moreover, $E\cap g\Lambda g^{-1}$ is Zariski dense in $\mathbb E.$   
\end{lem}

\begin{proof}
First note that thanks to Lemma~\ref{l;group-scheme} 
we may assume that $D$ is the $\kt$-points of a $\kt$-variety.
Now since the Zariski topology is Noetherian we may and will 
assume that $\mathbb D$ is minimal -- in the sense that
$\mu(\pi(\mathbb D'(\kt)))=0$ for all proper $\kt$-varieties $\mathbb D'\subset\mathbb D.$
This, in particular, implies that $D$ is irreducible.

Put $B':=\{g\in B:g\pi(D)=\pi(D)\}.$
The minimality assumption, and applying Lemma~\ref{l;group-scheme} again if necessary, 
imply that $\mu(\pi(D)\setminus g\pi(D))=0$ for all $g\not\in B'.$ 

Since $\mu$ is a probability measure, we get that $B'$ has finite index in $B.$ 
Let now $g\in B'$. Then $gD\subset D\Lambda$ and since $\Lambda$ is countable,
we get that there exists some $\lambda\in\Lambda$ so that $\mu(\pi(gD\cap D\lambda))>0.$   
Using our minimality assumption and Lemma~\ref{l;group-scheme} one more time, 
we get that $gD\subset D\lambda$; therefore, $gD=D\lambda.$ 
Since $\Lambda$ is countable, we get that $B'/B''$ is countable where $B'':=\{g\in B':gD=D\}.$ 

All together we have shown that $B''$ has countable index in $B.$ 
Recall that $B=\langle B_i\rangle$ where each $B_i=\mathbb B_i(\kt)$
and $\mathbb B_i$ is either a one dimensional $\kt$-split 
unipotent algebraic subgroup, or a one dimensional $\kt$-split tori;
in particular, $\mathbb B_i$ is connected for all $i$. 
This implies $B_i\cap B''$ is Zariski dense in $B_i.$ 
Hence $B_iD=D$ for all $i,$ which implies $BD=D.$ 
 
Put $\Lambda_0=\{\lambda\in\Lambda: D\lambda=D\},$ 
and $Y=D\setminus\cup_{\Lambda\setminus\Lambda_0}D\lambda.$ 
Minimality of $D$ implies that $\mu(\pi(Y))=\mu(\pi(D)).$
Moreover, $BY\Lambda_0=Y$ and the natural map $Y/\Lambda_0\to M/\Lambda$ is injective.
We thus get a $B$-invariant ergodic probability measure $\mu_0$ on $Y/\Lambda_0.$

Let $\mathbb F$ be the Zariski closure of $\Lambda_0$ in $\mathbb M$.
Then $\bbf$ is a $\kt$-group, see~\cite[Lemma 11.2.4(ii)]{Sp}, and $\mathbb D\mathbb F=\mathbb D.$ 
The push forward of $\mu_0$ gives a $B$-ergodic invariant measure 
on $D/F\subset(\mathbb D/\mathbb F)(\kt)$     
where $F=\mathbb F(\kt).$
Now by~\cite[Thm.~1.1 and Thm.~3.6]{Shalom}, this measure is the Dirac mass at one point.
That is there is some $z\in D$ so that $\mu(\pi(zF))=1.$ Since $zF\subset D,$
our minimality assumption implies $zF=D;$ in particular, we have $\bbf$ is connected. 
Therefore, $g=z$ and $\mathbb E=g\mathbb Fg^{-1}$ satisfy the claims in the lemma. 
\end{proof}

Let the notation be as in the beginning of \S\ref{sec;alg-measure}. 
We will say a Borel probability measure $\mu$ on $M/\Lambda$ is {\it Zariski dense} 
if there is no proper $\kt$-closed subset $\mathbb{M}$ of ${\bbm}$ 
such that $\mu(\pi(M))>0,$ where $M=\mathbb{M}(\kt).$ 
Two $\kt$-subvarieties $\bbl_1$ and $\bbl_2$ of ${\bbm}$ are said to be transverse at $x$ 
if they both are smooth at $x$ and 
\[
\mbox{$T_x(\bbl_1)\oplus T_x(\bbl_2)=T_x({\bbm}),$}
\] 
where $T_x(\bullet)$ denotes the tangent space of $\bullet$ at $x.$
Thanks to Lemma~\ref{l;group-scheme}, we also have the following, see~\cite[Prop.~3.3]{MT}.

\begin{lem}\label{zd-measure2}
Suppose
$B=\bbb(\kt)$ for a $\kt$-subgroup $\bbb$ of ${\bbm}.$
Assume that $\mu$ is a Zariski dense $B$-invariant Borel probability measure on $M/\Lambda$. 
Suppose $\mathbb{L}$ is a connected $\kt$-algebraic subvariety of ${\bbm}$ containing $e$ 
which is transverse to $\bbb$ at $e.$ 
Let $\mathbb D\subsetneq\bbl$ be a $\kt$-closed subset of $\bbl$ containing $e.$ 
Then, there exists a constant $0<\vare<1$ so that the following holds. 
If ${\Omega}\subset M/\Lambda$ is a measurable set with $\mu({\Omega})>1-\vare,$ 
then one can find a sequence $\{g_n\}$ of elements in $M$ with the following properties 
\begin{itemize}
\item[(i)] $\{g_n\}$ converges to $e,$ 
\item[(ii)] $g_n{\Omega}\cap{\Omega}\neq\emptyset,$ and 
\item[(iii)] $\{g_n\}\subset\bbl(\kt)\setminus\mathbb D(\kt).$ 
\end{itemize}   
\end{lem}

\subsection{Homogeneous measures.}\label{sec;homo-measure}
Let $M$ be a locally compact second countable group
and let $\Lambda$ be a discrete subgroup of $M$. 
Suppose $\mu$ is a Borel probability measure on $M/\Lambda$. 
Let $\Sigma=\{g\in M: g_*\mu=\mu\}$. 
The measure $\mu$ is called {\it homogeneous} if there exists $x\in M/\Lambda$ such that $\Sigma x$ is closed 
and $\mu$ is the $\Sigma$-invariant probability measure on $\Sigma x.$

\begin{lem}[Cf.~\cite{MT}, Lemma 10.1]\label{normal-unimodular}
Let $M$ be a locally compact second countable group and $\Lambda$ a discrete subgroup of $M.$ 
Suppose $B$ is a normal and unimodular subgroup of $M.$ 
Then any $B$-invariant, ergodic measure on $M/\Lambda,$ 
is homogeneous. Moreover, $\Sigma=\overline{B\Lambda}$.
\end{lem}

\subsection{Modulus of conjugation}\label{sec;modulus}
Suppose ${\bbm}$ is a $\kt$-algebraic group. Then $M={\bbm}(\kt)$ is a locally compact group. 
Let $a\in M$ and suppose $B\subset M$ is a closed, with respect to the Hausdorff topology, 
subgroup which is normalized by $a.$ 
We let $\alpha(a, B)$ denote the modulus of the conjugation action of $a$ on $B,$ 
i.e. if $Y\subset B$ is a measurable set 
\[
\mbox{$\theta(aYa^{-1})=\alpha(a, B)\theta(Y)$}
\]
where $\theta$ denotes a Haar measure on $B.$ 

Note that $\alpha(\cdot, M)$ is the modular function of $M.$ 
In particular, if $a\in [M,M],$ the commutator subgroup of $M,$ then $\alpha(a, M)=1$.

\section{Structure of pseudo reductive groups }\label{sec:psd-red}
In this section we will record a corollary of the main results in~\cite{CGP}; 
this section is required  only for our study in \S\ref{sec:arithmetic}.

We begin by fixing some notation. 
Let $\field$ be a local field. Throughout this section we let $\bbm$ be a connected, 
simply connected, semisimple group defined over $\field.$   
Moreover, we assume that either ${\rm char}(\field)>3$ or if ${\rm char}(\field)=2,3,$ 
then all of the absolutely almost simple factors of $\bbm$ are of type $A$.

Let $\mathsf B\subset\bbm(\field)$ be a Zariski dense subgroup of $\bbm$. 
Let $\sfield\subset\field$ be a closed subfield and put
\[
\bbf:=\begin{array}{c}\text{The irreducible component of the}\\ \text{ identity in the Zariski closure of $\mathsf B$ in }\mathcal R_{\field/\sfield}(\bbm)\end{array}
\] 

We will investigate the structure of $\bbf$ in this section.

\begin{lem}\label{lem:f-m}
\begin{enumerate}
\item ${\bbf}$ is a connected $\sfield$-subgroup of $\mathcal R_{\field/\sfield}(\bbm).$
\item $\bbf(\sfield)\subset\bbm(k)$ and $\bbf(\sfield)$ is Zariski dense in $\bbm.$ 
\end{enumerate}
\end{lem}
\begin{proof} 
The group $\bbf$ is connected by its definition.
Also since
\[
\mathsf B\subset\bbm(\field)=\mathcal R_{\field/\sfield}(\bbm)(\sfield),
\] 
the Zariski closure of $\mathsf B$ in $\mathcal R_{\field/\sfield}(\bbm)$
is defined over $\sfield,$ see~\cite[Lemma 11.2.4(ii)]{Sp}. This implies part (1).

The first claim in part~(2) follows from the definition.
To see the second claim, note that $\bbm$ is connected and $\mathsf B\cap\bbf(\sfield)$ has finite index in $\mathsf B.$   
\end{proof}

Put $\bbf':=[\bbf,\bbf],$ the commutator subgroup of $\bbf.$ 

\begin{lem}\label{lem:F-pseudo-red}
\begin{enumerate}
\item $\bbf'$ is a connected $\sfield$-group. 
\item $\bbf'(\sfield)$ is Zariski dense $\bbm.$
\item $\bbf$ is an $\sfield$-pseudo reductive group. That is $R_{u,\sfield}(\bbf)=\{e\}$ where $R_{u,\sfield}(\bbf)$ is the largest connected, normal, unipotent $l$-subgroup.
\item $\bbf'$ is a perfect group. That is it equals its own commutator group. 
\end{enumerate}
\end{lem}

\begin{proof}
Part~(1) follows from~\cite[\S 2.1]{B1}.

Part (2) is a consequence of Lemma~\ref{lem:f-m}(2) and~\cite[Ch.~I, \S2.1(e)]{B1} since $\bbm$ is semisimple.

Part (3) is a special case of a more general statement~\cite[Prop.\ 4.2.5]{CGP}. 
In the case at hand, however, a simpler proof is available as we now explicate.  Contrary to our claim, suppose that $R_{u,\sfield}(\bbf)$ is nontrivial. Then by~\cite[Ch.~1, Prop.\ 2.5.3]{Mar5} we have $R_{u,\sfield}(\bbf)(\sfield)$ is Zariski dense in $R_{u,\sfield}(\bbf)$; in particular, $R_{u,\sfield}(\bbf)(\sfield)$ is nontrivial. 
This, in view of Lemma~\ref{lem:f-m}(2), implies that $\bbm$ has a nontrivial unipotent radical, contradicting the fact
that $\bbm$ is semisimple. 

Part (4) follows from part (3) and~\cite[Prop.\ 1.2.6]{CGP}. 
\end{proof}

The following is the main result of this section; the proof is based on~\cite[Thm.~1.5.1 and Thm.~5.1.1]{CGP}.

\begin{thm}\label{thm:M-semisimple}
Let the notation be as above. Then,
\begin{itemize}
\item[(a)] there is a subfield $\sfield\subset\sfield'\subset \field$ with $\field/\sfield'$ a separable extension, and $\sfield'/\sfield$ a purely inseparable extension, 

\item[(b)] there is some $m\geq 1$ and for every $1\leq i\leq m$, there is a field $\sfield\subset\sfield_i\subset \sfield'$, in particular, $\sfield_i/\sfield$ is a purely inseparable extension,

\item[(c)] for every $1\leq i\leq m,$ there is an $\sfield_i$-simple, connected, simply connected, $\sfield_i$-group 
${\bbl}_i,$ and

\item[(d)] there is an isomorphism 
$
\iota: \prod_{i=1}^m\bbl_i\times_{\sfield_i}\sfield'\to\bbl,
$
where $\bbl$ is the irreducible component of the identity in the 
Zariski closure of $\mathsf B$ in $\mathcal R_{\field/\sfield'}(\bbm),$
\end{itemize}
so that the following hold.

\begin{enumerate} 
\item $\bbf'(\sfield)=\iota\Bigl(\prod_{i=1}^m{\bbl}_i(\sfield_i)\Bigr),$
 
\item $\bbf(\sfield)/\bbf'(\sfield)$ is a compact, abelian group. 
\end{enumerate}
\end{thm}

\begin{proof}
We give the proof in some steps. 

{\em Step 1.\ The field $\sfield'$.}\
Let $\sfield'=\sfield^{1/p^n}$ where $n$ is the largest positive integer so that $\sfield^{1/p^n}\subset\field$. 
Then $\sfield/\sfield'$ is purely inseparable and $\field/\sfield'$ is a separable extension.

Since $\mathsf B$ is Zariski dense in $\bbm$ and $\field/\sfield'$ is separable, 
we have that $\bbl$ is a connected, 
simply connected, semisimple group defined over $\sfield'.$   
Moreover, if ${\rm char}(\field)=2,3,$ 
then all of the absolutely almost simple factors of $\bbl$ are of type $A$.

Note also that since $\bbf$ is connected, we have
$
\bbf\subset\mathcal R_{\sfield'/\sfield}(\bbl).
$

{\em Step 2.\ The structure of $\bbf$}.\ Recall from Lemma~\ref{lem:F-pseudo-red}(3) 
that $\bbf$ is $\sfield$-pseudo reductive.
The map $\bbf\to \bbf/R_u(\bbf)$ factors as
\[
\bbf\hookrightarrow\mathcal R_{\sfield'/\sfield}(\bbl)\xrightarrow{q}\bbl
\] 
where the map $q$ is the natural projection, see~\cite[\S A.5]{CGP}. 
Therefore, if ${\rm char}(\field)=2,3,$ 
then all of the absolutely almost simple factors of $\bbf/R_u(\bbf)$ are of type $A.$  

The structure theory of $\bbf$ is extensively studied in~\cite{CGP}. 
The following is a corollary of~\cite[Thm. 1.5.1 and Thm.\ 5.1.1]{CGP} in the case at hand.
The group $\bbf$ is a standard pseudo reductive group, see~\cite[Def.\ 1.4.4]{CGP}. That is: 
\begin{itemize}
\item for all $1\leq i\leq m$, there is a (finite) purely inseparable extension $\sfield_i/\sfield,$
\item for all $1\leq i\leq m,$ there is an $\sfield_i$-simple, connected, simply connected, $\sfield_i$-group ${\bbl}_i,$
and a maximal $\sfield_i$-torus $\mathbb T_i,$
\item a commutative pseudo reductive $\sfield$-group $\bba$ which gives a factorization 
\[
\prod_{i=1}^m\mathcal R_{\sfield_i/\sfield}(\mathbb T_i)\xrightarrow{\phi} \mathbb A\to\prod_{i=1}^m\mathcal R_{\sfield_i/\sfield}(\bar{\mathbb T}_i),
\]
of the natural map induced from $\mathbb T_i\to\bar{\mathbb T}_i:=\mathbb T_i/Z(\bbl_i),$
\end{itemize}
so that $\bbf$ is isomorphic to
\be\label{eq:F-standard-rep}
\prod_{i=1}^m\mathcal R_{\sfield_i/\sfield}(\bbl_i)\rtimes\mathbb A/\prod_{i=1}^m\mathcal R_{\sfield_i/\sfield}(\mathbb T_i).
\ee
Moreover, $\bbf'=[\bbf,\bbf]$ is the image of the natural map  
\be\label{eq:F'-j}
j:\prod_{i=1}^m\mathcal R_{\sfield_i/\sfield}(\bbl_i)\to\prod_{i=1}^m\mathcal R_{\sfield_i/\sfield}(\bbl_i)\rtimes\mathbb A/\prod_{i=1}^m\mathcal R_{\sfield_i/\sfield}(\mathbb T_i),
\ee
and $\ker(j)=\ker(\phi)$ is central, see~\cite[Prop.\ 4.1.4 and Cor.\ A.7.11 ]{CGP}.

{\em Step 3.\ The proofs.}\ We now show that the collections $\{\sfield_i\}$ and $\{\bbl_i\}$ 
satisfy the claims in the proposition.

For all $1\leq i\leq m$, let us denote by $f_i:\mathcal R_{\sfield_i/\sfield}(\bbl_i)\times_{\sfield}\sfield'\to\bbl$ the map
\[
\mathcal R_{\sfield_i/\sfield}(\bbl_i)\times_{\sfield}\sfield'\xrightarrow{j}\bbf\times_{\sfield}\sfield'\hookrightarrow\mathcal R_{\sfield'/\sfield}(\bbl)\times_{\sfield}\sfield'\xrightarrow{q}\bbl
\]
where $q$ is the natural map.

Since $\sfield'/\sfield$ is purely inseparable, we have $\ker(q)$ is unipotent, see~\cite[Prop.\ A.5.11]{CGP}. 
Also recall that $\ker(j)$ is central. 
Therefore, $\ker(f_i)$ is a solvable group for all $i.$ Put 
\[
f=\prod_i f_i:\prod_{i=1}^m\mathcal R_{\sfield_i/\sfield}(\bbl_i)\times_{\sfield}\sfield'\to\bbl
\]
Then by Lemma~\ref{lem:F-pseudo-red}(2) and~\eqref{eq:F'-j}, $f$ is surjective with a solvable kernel.

Recall that $\sfield_i/\sfield$ is a purely inseparable extension for all $1\leq i\leq m.$
Therefore, $\ker(q_i)=R_u\Bigl(\mathcal R_{\sfield_i/\sfield}(\bbl_i)\Bigr)$ where
$q_i:\mathcal R_{\sfield_i/\sfield}(\bbl_i)\to\bbl_i$ is the natural surjection, see~\cite[Prop.\ A.5.11]{CGP}.
Moreover, $\bbl$ is a semisimple $\sfield'$-group, hence, 
\[
\ker(f_i)\supset R_u\Bigl(\mathcal R_{\sfield_i/\sfield}(\bbl_i)\Bigr).
\] 
Since $\ker(f_i)$ is a solvable and hence proper subgroup, we get from~\cite[Prop.\ A.7.8]{CGP}(1)
that
\[
\text{$\sfield_i\subset\sfield'$ for all $1\leq i\leq m;$}
\] 
here we also used the uniqueness of embedding of a purely inseparable extension.

As we mentioned above,
\[
\mathcal R_{\sfield_i/\sfield}(\bbl_i)\times_{\sfield}\sfield'/R_u\Bigl(\mathcal R_{\sfield_i/\sfield}(\bbl_i)\times_{\sfield}\sfield'\Bigr)=\bbl_i.
\]
Therefore, the map $f$ factors through a surjection
\[
\iota:\prod_i\bbl_i\times_{\sfield_i}\sfield'\to\bbl
\]
with central kernel. Recall, however, that $\bbl$ is simply connected; 
therefore, $\iota$ is an isomorphism.
This establishes claims (a)--(d). Note also that we get: $j$ is an isomorphism.

To see part~(1), note that 
\begin{align*}
\bbf'(\sfield)&=j\biggl(\textstyle\prod_{i=1}^m\mathcal R_{\sfield_i/\sfield}(\bbl_i)\biggr)(\sfield)&&\text{by~\eqref{eq:F'-j}}\\
&=j\biggl(\textstyle\prod_{i=1}^m\mathcal R_{\sfield_i/\sfield}(\bbl_i)(\sfield)\biggr)&&{\text{since $j$ is an isomorphism}}\\
&=j\Bigl(\textstyle\prod_{i=1}^m\bbl_i(\sfield_i)\Bigr)&&\text{since $R_{u,\sfield}\Bigl(\mathcal R_{\sfield_i/\sfield}(\bbl_i)\Bigr)=\{1\}$}\\
&=\iota\Bigl(\textstyle\prod_{i=1}^m\bbl_i(\sfield_i)\Bigr).
\end{align*}

We now turn to the proof of part (2).
Indeed, $\bbf(\sfield)/\bbf'(\sfield)$ is abelian as $\bbf/\bbf'$ is abelian.
To se it is also compact, consider $\bbf(\sfield)$ as a subgroup of $\bbl(\sfield').$ Then for every
$g\in \bbf(\sfield)$, conjugation by $g$ defines an automorphism of  $\bbl(\sfield')$ which preserves 
$\prod_i\bbl_i(\sfield_i).$
This automorphism is fiberwise an $\sfield_i$-algebraic automorphism. 
Since $\bbl_i(\sfield_i)$ is Zariski dense in $\bbl_i$, we get that this automorphism is fiberwise 
an inner $\sfield_i$-automorphism. Part (4) follows from this, in view of~\cite[Ch.~1, Thm.\ 2.3.1]{Mar5}.  
 \end{proof}

\section{An algebraic statement}\label{sec;algeb-lemma}
One of the remarkable features of Ratner's theorems is that 
they connect objects which are closely connected to the Hausdorff topology of the underlying group, like closure of a unipotent orbit or a measure invariant by a unipotent group, 
to objects which are described using the Zariski topology, e.g.\ algebraic subgroups. 
In positive characteristic setting these two topologies are rather far from each other\footnote{See~\cite{Pink} where structure of compact subgroups of semisimple groups of
adjoint type is described.}. 
From a philosophical stand point, this is one reason why the existing proofs in characteristic 
zero do not easily generalize to this case. 

Let us restrict ourselves to unipotent groups in order to highlight one of the differences. 
In characteristic zero, the group generated by one unipotent matrix already carries quite
a lot of information, e.g.\ it is Zariski dense in a one dimensional group. In positive characteristic,      
however, all unipotent elements are torsion. 
The situation improves quite a bit in the presence of a split torus action. In a sense, 
such an action can be used ``to redefine a notion of Zariski closure"
for the group generated by one element.
This philosophy is used in this section where we prove the following proposition
which is of independent interest. 

\begin{prop}\label{p:TorusInvariantUnipotent}
Let $k$ be a local field of characteristic $p>0$ and $\bbw$ 
a unipotent $k$-group equipped with a $k$-action by $\GL_1$. 
Assume that all the weights are positive integers. 
Let $\ucal$ be a subgroup of $\bbw(k)$ which 
is invariant under the action of $k^{\times}$. 
Then, there is $q$ a power of $p,$ which only depends on the set of weights, 
such that $\ucal=\bbw'(k^q)$, where $\bbw'$ is the Zariski-closure of 
$\ucal\subset \bbw(k)=\rcal_{k/k^q}(\bbw)(k^q)$ in $\rcal_{k/k^q}(\bbw)$.
Moreover, $\bbw'$ is connected.
\end{prop}
Groups like $\ucal$ arise naturally in our study, see \S\ref{sec;proof} for more details. 

We shall prove Proposition~\ref{p:TorusInvariantUnipotent} in several steps. 
First we prove it when $\bbw$ is a commutative $p$-torsion $k$-group. 
In the next step, the general commutative case is handled. 
In the final step, the general case is proved by induction on the nilpotency length. 

In order to prove the first step, we shall start with a few auxiliary lemmas. 
In Lemmas \ref{l:PowerPWeight1} and \ref{l:PowerPWeight2}, 
we assume that the weights are powers of $p$. 
In Lemmas \ref{l:Separable} and \ref{l:SplitWeightSpaces}, 
we get a convenient decomposition of $\ucal$ into certain subgroups, 
and finally in Lemma \ref{l:VectorGroupTorusAction}, we prove the first step.
Let begin with the following

\begin{lem}\label{l:PowerPWeight1}
Let $F$ be an infinite field of characteristic $p$ and  $0<l_1<\cdots<l_n$ positive integers. 
Assume $\GL_1$ acts linearly on a standard $F$-vector group $\bbw$ with weights equal to $p^{l_i}$. 
Let $\bbw_i$ be the weight space of $p^{l_i}$ and 
suppose $\ucal$ is a subgroup of $\bbw(k)$ which is invariant under $\GL_1(k)$. 
If $\ucal$ does not intersect $\oplus_{i=2}^{n}\bbw_i$, then 
\[
\ucal=\bbw'(F^{p^{l}}),
\]
for any $l\ge l_1$, where $\bbw'$ is the Zariski-closure of $\ucal$ in $\rcal_{F/F^{p^l}}(\bbw)$.
\end{lem}
\begin{proof}
Via the action of $\GL_1(F),$ and in view of our assumption of the weights, 
we can view $\bbw(F)$ as an $F$-vector space and $\ucal$ as a $F$-subspace. 
Since $\ucal$ does not intersect $\oplus_{i=2}^n \bbw_i$, 
we get an $F$-linear map $\theta$ from ${\rm pr}_1(\ucal)$ to $\oplus_{i=2}^n \bbw_i$, 
where ${\rm pr}_1:\bbw\rightarrow \bbw_1$ is the projection map, and we have
\[
\ucal=\{(\xbf,\theta(\xbf))|\h\xbf\in {\rm pr}_1(\ucal)\}.
\]
It is clear that ${\rm pr}_1(\ucal)$ is a $F^{p^{l_1}}$-subspace of $\bbw_1(F)$ 
with respect to the standard scalar multiplication. 
It is also clear that $\theta$ can be extended to an $F$-morphism from $\bbw_1$ to $\oplus_{i=2}^n\bbw_i$. 
Hence there is a standard $F^{p^{l_1}}$-vector subgroup $\bbw_1'$ of $\rcal_{F/F^{p^{l_1}}}(\bbw_1)$ such that
\[
\ucal=\{(\xbf,\ybf)|\h\xbf \in \bbw'_1(F^{p^{l_1}}),\h\ybf=\rcal_{F/F^{p^{l_1}}}(\theta)(\xbf)\}.
\]
Since $F$ is an infinite field, the Zariski-closure $\bbw'$ of $\ucal$ in $\rcal_{F/F^{p^{l_1}}}(\bbw)$ is equal to
\[
\{(\xbf,\ybf)|\h \xbf\in \bbw_1', \h\ybf=\rcal_{F/F^{p^{l_1}}}(\theta)(\xbf)\},
\]
which shows that $\ucal=\bbw'(F^{p^{l_1}})$. Now, one can easily deduce the same result for any $l\ge l_1$.
\end{proof}

\begin{lem}\label{l:PowerPWeight2}
Let $F$ be an infinite field of characteristic $p$ and  $0<l_1<\cdots<l_n$ positive integers. 
Assume $\GL_1$ acts linearly on a standard $F$-vector group $\bbw$ with weights equal to $p^{l_i}$. 
Let $\ucal$ be a subgroup of $\bbw(F)$ which is invariant under $\GL_1(F)$. Then 
\[
\ucal=\bbw'(F^{p^{l}}),
\]
for any $l\ge  l_n$, where $\bbw'$ is the Zariski-closure of $\ucal$ in $\rcal_{F/F^{p^l}}(\bbw)$.
\end{lem}

\begin{proof}
We denote the weight space of $p^{l_i}$ by $\bbw_i$. 
Let $\ucal_i=\ucal\cap (\oplus_{j=i}^n \bbw_j)$ and define $\ucal_i'$ 
to be a $\GL_1(F)$-invariant complement of 
$\ucal_i$ in $\ucal_{i-1}$. So we have
\begin{align*}
&\ucal_i'\subset \oplus_{j=i-1}^n \bbw_j\\
&\ucal_i'\cap (\oplus_{j=i}^n \bbw_j)= \{0\}\\
&\ucal_{i-1}=\ucal_{i}\oplus \ucal_i'\\
&\ucal=\ucal_2'\oplus \ucal_3' \oplus\cdots \oplus \ucal_{n}'\oplus \ucal_n.
\end{align*}
By Lemma \ref{l:PowerPWeight1}, we have that $\ucal_i'=\bbw_i'(F^{p^l})$, for all $i$ and any $l\ge l_{i-1}$, 
where $\bbw_i'$ is the Zariski-closure of $\ucal_i'$ in $\rcal_{k/k^{p^l}}(\oplus_{j=i-1}^n \bbw_j)$. 
Moreover, $\ucal_n$ is a subspace of $\bbw_n$ with respect to the standard action of $F^{p^n};$ 
one can easily conclude.
\end{proof}
\begin{lem}\label{l:Separable}
Let $m_1,\ldots,m_d$ be distinct positive integers which are coprime with $p$. 
Let $g(x)=(x^{m_1},\ldots,x^{m_d})$ be a morphism from $\bba^{1}$ to $\bba^d$. Then 
\[
G(x_1,\ldots,x_d):=g(x_1)+\cdots+g(x_d)
\]
is a separable function from $\bba^{d}$ to $\bba^{d}$ at a $F$-point, 
for any infinite field $F$ of characteristic $p$.
\end{lem}
\begin{proof}
It is enough to show that the Jacobian of $G$ is invertible at some $F$-point. 
Thus, thanks to our assumption: all of $m_i$ are coprime with $p$, 
it suffices to show that the kernel of $D=[x_j^{m_i-1}]$ is trivial for some $x_j\in F$.  
Now since $F$ is an infinite field, there is an element $x\in F^{\times}$ of multiplicative order larger than $\max_i m_i$. 
Set $x_j=x^{j-1};$ if $D$ has a non-trivial kernel, then there is non-zero polynomial $Q$ 
of degree at most $d-1$ with coefficients in $F$, such that
\[
Q(x^{m_1-1})=Q(x^{m_2-1})=\cdots=Q(x^{m_d-1})=0.
\]
This is a contradiction as $x^{m_i-1}$ are distinct and the degree of 
$Q$ is at most $d-1$. 
\end{proof}

\begin{lem}\label{l:SplitWeightSpaces}
Let $k$ be a local field of characteristic $p$ and $m_1,\cdots,m_d$ distinct positive integers which are coprime with $p$. 
Let $\bbw$ be a standard $k$-vector group equipped with a linear $k$-action by $\GL_1$. 
Assume that the set of weights 
$\Phi= \Phi_1\sqcup \cdots \sqcup\Phi_d$, $\Phi_i=\{p^lm_i\in\Phi|\h l\in\bbn\},$ and moreover
$\Phi_i$ is non-empty. Let $\ucal$ be a subgroup of $\bbw(k)$, 
which is invariant under the action of $\GL_1(k)$. Then 
\[
 \ucal=\ucal_1\oplus \cdots \oplus \ucal_d,
\]
where $\ucal_i=\ucal\cap (\oplus_{\alpha\in \Phi_i} \bbw_\alpha)$ and $\bbw_\alpha$ 
is the weight space corresponding to $\alpha$. 
Furthermore, if $\Phi_i=\{p^{l_{i1}}m_i,p^{l_{i2}}m_i,\ldots, p^{l_{in_i}}m_i\}$ and $\xbf=(x_1,\ldots,x_{n_i})\in \ucal_i$, then 
\[
 \mbox{$(\lambda^{p^{l_{i1}}} x_1,\ldots, \lambda^{p^{l_{in_i}}} x_{n_i})\in \ucal_i,\;$ 
 for any $\lambda\in k^{\times}$.}
\]
\end{lem}

\begin{proof}
Take an arbitrary element $\xbf=(x_\alpha)_{\alpha\in \Phi}\in \ucal$. 
Since $\ucal$ is invariant under the action of $\GL_1(k)$, we have
$(\lambda^\alpha x_\alpha)_{\alpha\in\Phi}$ is also in $\ucal$, for any $\lambda\in k^{\times}$. 
On the other hand, as $\ucal$ is a group, 
\begin{align}
\label{e:Weight}((\pm\lambda_1^\alpha\pm\lambda_2^\alpha\pm\cdots\pm\lambda_n^\alpha)x_\alpha)_{\alpha\in\Phi}\in \ucal,
\end{align}
for any $\lambda_1,\ldots,\lambda_n\in k^{\times}$. 
On the other hand, by Lemma~\ref{l:Separable} and the Inverse Function Theorem, 
the image of $G-G$ has an open neighborhood of the origin. Therefore, thanks
to scale invariance of the image, we have:  
for any $\lambda_1',\ldots,\lambda_d' \in k$, one can find $\lambda_i,\mu_i\in k$ such that
\begin{equation}
\label{e:Surjective}\mbox{$\lambda_i'=\sum_{j=1}^d \lambda_j^{m_i}-\sum_{j=1}^d\mu_j^{m_i},$}
\end{equation}
for any $1\le i\le d$. 
Now since $p$ is the characteristic of $k,$ using \eqref{e:Weight} and \eqref{e:Surjective} 
one can easily finish the argument.
\end{proof}

\begin{lem}\label{l:VectorGroupTorusAction}
Let $k$ be a local field of characteristic $p$. 
Let $\bbw$ be a $p$-torsion commutative unipotent $k$-group equipped 
with a linear $k$-action by $\GL_1;$ further, 
assume that all the weights are positive. 
Let $\ucal$ be a subgroup of $\bbw(k)$, which is invariant under the action of $\GL_1(k)$. 
Then, there exists some $l_0,$ depending only on the weights, 
such that for any integer $l\ge l_0$
\[
\ucal=\bbw'(k^{p^{l}}),
\]
where $\bbw'$ is the Zariski-closure of $\ucal$ in $\rcal_{k/k^{p^l}}(\bbw)$.
\end{lem}

\begin{proof}
By \cite[Prop.~B.4.2]{CGP}, we can assume that $\bbw$ is a 
$k$-vector group equipped with a $k$-linear action of $\GL_1$. 
Applying Lemma~\ref{l:SplitWeightSpaces}, we can decompose $\ucal$ into subgroups 
$\ucal_i$ and get a new $\GL_1$ action on $\oplus_{\alpha\in \Phi_i} \bbw_\alpha$ such that 
all the new weights are powers of $p$ and $\ucal_i$ is invariant under this new action of $\GL_1(K)$. 
The lemma now follows from Lemma \ref{l:PowerPWeight2}.
\end{proof}

\begin{lem}\label{l:CommutativeTorusAction}
Let $k$ be a local field of characteristic $p$. 
Let $\bbw$ be a commutative unipotent $k$-group equipped with a 
$k$-action by $\GL_1$ such that $Z_{\GL_1}(\bbw)=\{1\}$. 
Suppose $\ucal$ is a subgroup of $\bbw(k)$, which is invariant under the action of $\GL_1(k)$. 
Then, there is some $l_0,$ depending only on the weights of the action of $\GL_1$ on $\Lie(\bbw),$ 
such that for any integer $l\ge l_0$
\[
\ucal=\bbw'(k^{p^{l}}),
\]
where $\bbw'$ is the Zariski-closure of $\ucal$ in $\rcal_{k/k^{p^l}}(\bbw)$.
\end{lem}

\begin{proof}
Since $\bbw$ is unipotent, it is a torsion group. 
We now proceed by induction on the exponent of $\bbw;$ 
if it is $p$, by Lemma~\ref{l:VectorGroupTorusAction}, we are done. 
Thus, assume that the exponent of $\bbw$ is $p^l$. 
Let $\ucal[p]=\{u\in \ucal| u^{p}=1\};$ 
since $\ucal$ is commutative, $\ucal[p]$ is a subgroup of $\ucal$ 
which is clearly invariant under the action of $\GL_1(k)$. 
Let $\bbv$ be the Zariski-closure of $\ucal[p]$ in $\bbw;$ 
then, $\bbv$ is a $p$-torsion commutative unipotent group 
which in view of Lemma~\ref{l;group-scheme} is defined over $k$. 
Therefore, by Lemma \ref{l:VectorGroupTorusAction} for a large enough power of $p$ (depending only on the weights), which we denote by $q',$ we have 
$\ucal[p]=\bbw^{(p)}(k^{q'})$, where $\bbw^{(p)}$ is the Zariski-closure of $\ucal[p]$ in 
$\rcal_{k/k^{q'}}(\bbv)\subset\rcal_{k/k^{q'}}(\bbw)$. 

Let $\bbw''$ be the Zariski-closure of $\ucal$ in $\rcal_{k/k^{q'}}(\bbw)$.
Note that $\GL_1$ acts on $\bbw''$ with no trivial weights, 
and $\bbw^{(p)}$ is invariant under this action. 
Hence both of these groups and their quotient are $k^{q'}$-split unipotent groups. 
We now consider the following exact sequence of $k^{q'}$-split unipotent groups,
\begin{align}
\notag 1\rightarrow \bbw^{(p)} \rightarrow \bbw'' \xrightarrow{\pi} \bbw''/\bbw^{(p)} \rightarrow 1. 
\end{align}
Note that $\pi(\ucal)$ is Zariski-dense in $\bbw''/\bbw^{(p)}$ and $p^{l-1}$-torsion, which implies  $\bbw''/\bbw^{(p)}$ is $p^{l-1}$-torsion. Hence, by induction hypothesis, there exists $q\geq q'$ which is a large enough power of $p,$ depending only on the weights, such that
\begin{align}
 \pi(\ucal)=\overline{\bbw}'(k^{q}),\label{e:Image}
\end{align} 
where $\overline{\bbw}'$ is the Zariski-closure of $\pi(\ucal)$ 
in $\rcal_{k^{q'}/k^{q}}(\bbw''/\bbw^{(p)})$. On the other hand, 
\begin{align}
\ucal[p]=\bbw^{(p)}(k^{q'})=\rcal_{k^{q'}/k^{q}}(\bbw^{(p)})(k^{q}) \label{e:Ker}
\end{align}
 and, by \cite[Cor.~A.3.5]{Oes}, $\rcal_{k^{q'}/k^{q}}(\bbw^{(p)})$ is 
 $k^q$-split unipotent group. Thus $\ucal[p]$ is Zariski-dense in 
 $\rcal_{k^{q'}/k^{q}}(\bbw^{(p)})$. By \cite[Prop.~A.3.8]{Oes}, we also know that
 the following is exact
 \begin{align}
\notag 1\rightarrow  \rcal_{k^{q'}/k^q}(\bbw^{(p)})\rightarrow \rcal_{k^{q'}/k^q}(\bbw'') 
\xrightarrow{\pi'} \rcal_{k^{q'}/k^q}(\bbw''/\bbw^{(p)}) \rightarrow 1.
 \end{align}
 Now let $\bbw'$ be the Zariski-closure of $\ucal$ in $\rcal_{k/k^q}(\bbw'')$. 
 By the above discussion, it is easy to get the following short exact sequence and 
 show that all of the involved groups are $k^q$-split unipotent groups,
 \begin{align}
 \notag 1\rightarrow  \rcal_{k^{q'}/k^q}(\bbw^{(p)}) \rightarrow \bbw' \xrightarrow{\pi} 
 \overline{\bbw}' \rightarrow 1.
 \end{align}
  By (\ref{e:Image}) and (\ref{e:Ker}) and the fact that these groups are 
  $k^q$-split unipotent groups, we get the following exact sequence,
 \begin{align}
  1\rightarrow  \ucal[p] \rightarrow \bbw'(k^q) \xrightarrow{\pi} \pi(\ucal) \rightarrow 1.\label{e:ExactSeq}
 \end{align}
 So, by (\ref{e:ExactSeq}) and $\ucal\subset \bbw'(k^q)$, 
 one can easily deduce that $\ucal=\bbw'(k^q)$, which finishes the proof.
\end{proof}

\begin{proof}[Proof of Proposition~\ref{p:TorusInvariantUnipotent}]
We proceed by induction on the nilpotency length of $\bbw$. If it is commutative, 
by Lemma \ref{l:CommutativeTorusAction}, we are done.
Assume $\bbw$ is of nilpotency length $c$. Then $[\ucal,\ucal]\subset [\bbw,\bbw](k)$, 
where $[\bullet,\bullet]$ is the derived subgroup of $\bullet$. 
The nilpotency length of $[\bbw,\bbw]$ is $c-1;$ hence, by induction hypothesis, 
for any $q'$ which is a large enough power of $p$ (depending only on the weights), we have
\begin{align}
\notag [\ucal,\ucal]=\widetilde{\bbw}(k^{q'}),
\end{align}
where $\widetilde{\bbw}$ is the Zariski-closure of $[\ucal,\ucal]$ in $\rcal_{k/k^{q'}}(\bbw)$. 
Let $\bbw''$ be the Zariski-closure of $\ucal$ in $\rcal_{k/k^{q'}}(\bbw);$ 
since $\GL_1$ acts on $\bbw''$ with no trivial weights and since $\widetilde{\bbw}$ 
is invariant under this action, both of these groups and the quotient group are $k^{q'}$-split groups. 
We consider the following short exact sequence
\begin{align}
1\rightarrow \widetilde{\bbw} \rightarrow \bbw'' \xrightarrow{\pi} \bbw''/\widetilde{\bbw} \rightarrow 1; \label{e:ExactDerive}
\end{align}
since $\pi(\ucal)$ is commutative and Zariski-dense in $\bbw''/\widetilde{\bbw}$, 
we get that $\bbw''/\widetilde{\bbw}$ is commutative. 
Therefore, by Lemma \ref{l:CommutativeTorusAction}, if $q\geq q'$ is a large enough power of $p$ 
(depending only on the weights), we have
\begin{align}
\pi(\ucal)=\overline{\bbw}'(k^q), \label{e:Abelianization}
\end{align}
where $\overline{\bbw}'$ is the Zariski-closure of $\pi(\ucal)$ in $\rcal_{k^{q'}/k^q}(\bbw''/\widetilde{\bbw})$. We also have 
\begin{align}
[\ucal,\ucal]=\widetilde{\bbw}(k^{q'})=\rcal_{k^{q'}/k^q}(\widetilde{\bbw}')(k^q).\label{e:DerivedSubgroup}
\end{align}
By \cite[Prop.~A.3.8]{Oes} and (\ref{e:ExactDerive}), we have the exact sequence
\begin{align}
1\rightarrow \rcal_{k^{q'}/k^q}(\widetilde{\bbw}) \rightarrow  
\rcal_{k^{q'}/k^q}(\bbw'') \xrightarrow{\pi'}  \rcal_{k^{q'}/k^q}(\bbw''/\widetilde{\bbw}) \rightarrow 1. \label{e:[U,U]}
\end{align}
Let $\bbw'$ be the Zariski-closure of $\ucal$ in $\rcal_{k^{q'}/k^q}(\bbw'')$. 
Since $\widetilde{\bbw}$ is a $k^{q'}$-split unipotent group, 
by (\ref{e:Abelianization}) and (\ref{e:[U,U]}), we have the following exact sequence
\begin{align}
\notag 1\rightarrow  \rcal_{k^{q'}/k^q}(\widetilde{\bbw})\rightarrow \bbw' \xrightarrow{\pi'} \overline{\bbw}'\rightarrow 1
\end{align}
and so, by (\ref{e:DerivedSubgroup}), (\ref{e:[U,U]}) and the fact that all the 
involved groups are $k^q$-split unipotent groups, we have
\begin{align}
1\rightarrow [\ucal,\ucal] \rightarrow \bbw'(k^q) \rightarrow \pi(\ucal) \rightarrow 1;
\end{align}
this together with $\ucal\subset\bbw'(k^q)$ finishes the proof except connectedness.

To see $\bbw'$ is connected, note that in view of our assumption 
that all the weights are positive, there exists some
$r\in k$ so that every element in $\ucal$ is contracted to the identity by $r.$  
\end{proof}





\section{Polynomial like behavior and the basic lemma}\label{sec;quasi}

In this section we assume $\mu$ is a probability measure on $X=G/\Gamma$ 
which is invariant under the action of some $\kt$-split, unipotent $\kt$-subgroup of $G.$ 

We will recall an important construction 
based on the slow divergence of two nearby unipotent orbits in $X$.
Then, we will use this to acquire new elements in the stabilizer of $\mu$. 
Investigating the polynomial like behavior of two diverging unipotent orbits in the {\em intermediate range} 
dates back to several important works, 
e.g.\ Margulis' celebrated proof of the Oppenheim conjecture~\cite{Mar2}, using topological arguments, 
and Ratner's seminal work on the proof of the measure rigidity conjecture~\cite{Rat2, Rat3, Rat4}. 

\subsection{Construction of quasi-regular maps}
This section follows the construction in~\cite[\S5]{MT}.
It is written in a more general setting than what is needed for the proof of Theorem~\ref{t;U-measure-class}, namely {\em $\mu$ is not assumed to be ergodic for the action of the unipotent group $\ucal$ 
which is used in the construction.}
We first recall the definition of a quasi-regular map. 
Here the definition is given in the case of a local field, which is what we need later, 
the $\mathcal{T}$-arithmetic version is a simple modification.  
It is worth mentioning that we have a simplifying assumption here compare to the situation in~\cite{MT}:
our group $\ucal$ is normalized and expanded by an element from class $\acal.$ This is
used in order to define nice Folner sets in $\ucal.$ 
In view of this, we do not need the construction of the group $U_0$ in~\cite{MT}.

\begin{definition}[Cf.~\cite{MT}, Definition 5.3]\label{quasiregular}
Let $k$ be a local field.
\begin{enumerate}
\item Let $\mathbb{E}$ be a $k$-algebraic group, 
$\mathcal{U}$ a $k$-subgroup of $\mathbb{E}(k),$ 
and $\mathbb{M}$ a $k$-algebraic variety. 
A $k$-rational map $f:\mathbb{M}(k)\rightarrow\mathbb{E}(k)$ is called 
$\mathcal{U}$-{\it quasiregular} if the map from $\mathbb{M}(k)$ to $\mathbb{V},$ 
given by $x\mapsto\rho(f(x))\qpz,$ is $k$-regular 
for every $k$-rational representation $\rho:\mathbb{E}\rightarrow\mbox{GL}(\mathbb{V}),$ 
and every point $\qpz\in\mathbb{V}(k)$ such that $\rho(\mathcal{U})\qpz=\qpz.$
\item Let $E=\mathbb{E}(k)$ and suppose $\mathcal{U}\subset E$ is a 
$k$-split unipotent subgroup. A map $\phi:\mathcal{U}\rightarrow E$ 
is called {\it strongly} $\mathcal{U}$-{\it quasiregular} if there exist
\begin{itemize}
\item[(a)] a sequence $g_n\in E$ such that $g_n\rightarrow e,$
\item[(b)]  a sequence $\{\alpha_n:\mathcal{U}\rightarrow\mathcal{U}\}$ of 
$k$-regular maps of bounded degree,
\item[(c)]  a sequence $\{\beta_n:\mathcal{U}\rightarrow\mathcal{U}\}$ of 
$k$-rational maps of bounded degree, and
\item[(d)] a Zariski open, dense subset $\mathcal{X}\subset\mathcal{U},$ 
\end{itemize}
such that $\phi(u)=\lim_{n\rightarrow\infty}\alpha_n(u)g_n\beta_n(u),$ 
and the convergence is uniform on the compact subsets of $\mathcal{X}.$
\end{enumerate}
\end{definition}


\noindent
We note that if $\phi$ is strongly $\mathcal{U}$-quasiregular, 
then it indeed is $\mathcal{U}$-quasiregular.
To see this, let $\rho:E\rightarrow\mbox{GL}(V)$ be a $k$-rational representation, 
and let $\qpz\in V$ be a $\mathcal{U}$-fixed vector. 
For any $u\in\mathcal{X}$ we have 
\begin{equation}\label{e;rho}
\rho(\phi(u))\qpz=\lim_{n\rightarrow\infty}\rho(\alpha_n(u)g_n)\qpz.
\end{equation}
Thanks to the fact that $\mathcal{U}$ is split we can identify $\mathcal{U}$ with an affine space.
Then 
\[
\mbox{$\psi_n:\mathcal{U}\rightarrow V\;$ given by $\psi_n(u)=\rho(\alpha_n(u)g_n)\qpz$} 
\]
is a sequence of polynomial maps of bounded degree. 
Moreover, this family is uniformly bounded on compact sets of $\mathcal{X}.$ 
Therefore, it converges to a polynomial map with coefficients in $k.$ 
This shows $\phi$ is $\mathcal{U}$-quasiregular.

For the rest of this section we assume the following
\begin{itemize}
\item $k$ is a local field, 
\item $G$ is the group of $k$-points of a $k$-group,
\item $\ucal$ is a connected $k$-split, unipotent $k$-subgroup of $G,$
\item there is an element $s\in G$ from class $\mathcal A$ so that 
$\ucal\subset W^+_G(s)$ and $\ucal$ is normalized by $s.$
\end{itemize}

In view of these assumptions,~\cite[Prop.~9.13]{BS} implies that there a 
regular cross section, $\mathcal V$, for $\ucal$ in $W^+_G(s)$ which is invariant under conjugation by $s$.
Put 
\[
L:=W^-_G(s)Z_G(s)\mathcal V.
\]
Then $L$ is a rational cross section for $\ucal$ in $G.$

We fix, $\mathfrak{B}^+$ and $\mathfrak{B}^-$, relatively compact neighborhoods of $e$ in 
${\wpg}$ and $\wmg$ respectively, with the property that 
\[
\mathfrak{B}^+\subset s\mathfrak{B}^+s^{-1}\text{ and } \mathfrak{B}^-\subset s^{-1}\mathfrak{B}^-s.
\] 
Using these, we define a filtration in ${\wpg}$ and ${\wmg}$ as follows 
\[
\mathfrak{B}_n^+= s^n\mathfrak{B}^+s^{-n}\text{ and } \mathfrak{B}_n^-= s^{-n}\mathfrak{B}^-s^{n}. 
\] 
Define $\ell^{\pm}:W^{\pm}(s)\rightarrow\bbz\cup\{-\infty\}$ by
\begin{itemize}
\item $\ell^+(x)=j$ if $x\in\mathfrak{B}_j^+\setminus\mathfrak{B}_{j-1}^+,$ and $\ell^+(e)=-\infty,$
\item $\ell^-(x)=j$ if $x\in\mathfrak{B}_j^-\setminus\mathfrak{B}_{j-1}^-,$ and $\ell^-(e)=-\infty.$
\end{itemize}
For any integer $n$, set $\ucal_n=\mathfrak{B}_n^+\cap\ucal.$  

Let $\{g_n\}\subset L\hh\ucal\setminus N_G(\ucal)$ be a sequence with $g_n\rightarrow e.$ 
Since $L$ is a rational cross-section for $\ucal$ in $G,$ we get rational morphisms 
\[
\tphi_n:\ucal\rightarrow L\text{ and }{\omega}_n:\ucal\rightarrow\ucal
\] 
such that $ug_n=\tphi_n(u){\omega}_n(u)$ holds for all $u$ in a Zariski open, dense subset of $\ucal.$ 

Recall that by a theorem of Chevalley, there exists a $\kt$-rational representation 
$\rho:G\rightarrow\GL(\Psi)$ and a unit vector $\qpz\in \Psi$ such that 
\begin{equation}\label{e;chevalley1}
\ucal=\{g\in G:\h\rho(g) \qpz=\qpz\}.
\end{equation}
According to this description we also have 
\begin{equation}\label{e;chevalley2}
\rho(N_G(\ucal))\qpz=\{\mathpzc{z}\in \rho(G)\qpz:\h\rho(\ucal)\mathpzc{z}=\mathpzc{z}\}.
\end{equation} 
Fix a bounded neighborhood $\bcal(\qpz)$ of $\qpz$ in $\Psi$ such that 
\begin{equation}\label{e;nbhd}
\rho(G)\qpz\cap\bcal(\qpz)=\overline{\rho(G)\qpz}\cap\bcal(\qpz),
\end{equation}
where the closure is taken with respect to the Hausdorff topology of $\Psi.$ 

Recall that $g_n\notin N_G(\ucal)$. Thus, in view of~\eqref{e;chevalley2},
there is a sequence of integers $\{b(n)\}$ such that 
\begin{itemize}
\item $b(n)\rightarrow\infty$, 
\item $\rho(\ucal_{b(n)+1}g_n)\qpz\not\subset\bcal(\qpz),$ and  
\item $\rho(\ucal_{m}g_n)\qpz\subset\bcal(\qpz)$ for all $m\leq b(n).$
\end{itemize}
Define $k$-regular isomorphisms $\tau_n:\ucal\rightarrow\ucal$ as follows. 
For every $u\in\ucal$ put
\begin{equation}\label{e;exp-quai}
\tau_n=\la_{b(n)}\text{ where }\lambda_n(u)=s^{n}us^{-n}.
\end{equation} 
Given $n\in\bbn,$ we now define the $k$-rational map 
$\phi_n:\ucal\rightarrow L$ by $\phi_n:=\tphi_n\circ\tau_n.$ 

Let $\rho_L$ be the restriction of the orbit map $g\mapsto \rho(g)\qpz$ to $L$  
and define 
\begin{equation}\label{e;quasi1}
\phi'_n:=\rho_L\circ\phi_n:\ucal\rightarrow \Psi.
\end{equation}
It follows from the definition of $b(n)$ that 
$\phi'_n(\mathfrak{B}_0)\subset\mathcal{B}(\qpz),$ but $\phi'_n(\mathfrak{B}_1)\not\subset\mathcal{B}(\qpz)$. 

Note that $\phi'_n(u)=\rho(\alpha_n(u)g_n)\qpz.$ 
Hence $\phi_n':\ucal\rightarrow \Psi$ is a $k$-regular morphism. 

Since 
\begin{itemize}
\item $\ucal$ is a connected $k$-group, 
\item $\ucal$ is normalized by $S,$ and 
\item $Z_G(S)\cap\ucal=\{e\}$
\end{itemize}
we get from~\cite[Cor.~9.12]{BS} that $\ucal$ and its Lie algebra are 
$S$-equivariantly isomorphic as $k$-varieties. 
Hence, $\{\phi_n'\}$ is a sequence of equicontinuous polynomials of bounded degree. 
Therefore, after possibly passing to a subsequence, we assume that 
there exists a $k$-regular morphism $\phi':\ucal\rightarrow \Psi$ such that
\begin{equation}\label{conv-eq}
\phi'(u)=\lim_{n\rightarrow\infty}\phi'_n(u)\mbox{ for every }u\in\ucal.
\end{equation}
The map $\phi'$ is non-constant since 
$\phi'(\overline{\mathfrak{B}_1})$ in not contained in $\mathcal{B}(\qpz);$ 
moreover, since $g_n\rightarrow e$ we have $\phi_n'(e)\rightarrow \qpz,$ hence, $\phi'(e)=\qpz.$ 

Let $\mcal=\rho(L)\qpz$. Since $L$ is a rational cross section for $\ucal$ in $G$ which contains $e,$ 
we get that $\mcal$ is a Zariski open dense subset of $\rho(G)\qpz$ and $\qpz\in\mcal.$ 

Let now $\phi:\ucal\rightarrow L$ be the $\kt$-rational morphism defined by
\begin{equation}\label{e;quasi2}
\phi(u):=\rho_L^{-1}\circ\phi'(u). 
\end{equation}
It follows from the construction that $\phi(e)=e$ and that $\phi$ is non-constant.

{\em Claim.} The map $\phi$ constructed above is strongly $\ucal$-quasiregular.

To see the claim, first note that by the definition of
$\phi_n$ and in view of~\eqref{conv-eq} and~\eqref{e;quasi2} we have 
\begin{equation}\label{e;quasi3}
\phi(u)=\lim_{n\rightarrow\infty}\phi_n(u)\text{ for all } u\in\phi'^{-1}(\mathcal{M}).
\end{equation}
Now since the convergence in \eqref{conv-eq} is uniform
on compact subsets and since $\rho_L^{-1}$ is continuous on compact subsets of $\mathcal{M},$ 
we get that the convergence in \eqref{e;quasi3} 
is also uniform on compact subsets of $\phi'^{-1}(\mathcal{M}).$ 
Recall that 
\[
\tau_n(u)g_n=\phi_n(u){w}_n(\tau_n(u)).
\] 
Hence, for any $u\in\phi'^{-1}(\mathcal{M})$ we can write
\begin{equation}\label{e;quasi4}
\phi(u)=\lim_{n\rightarrow\infty}\tau_n(u)g_n({w}_n(\tau_n(u)))^{-1};
\end{equation} 
the claim follows.

\subsection{Properties of quasi-regular maps and the Basic Lemma}\label{sec:using-quasi-reg}
We will need some properties of the map $\phi$ constructed above. 
The proofs of these facts are mutandis mutatis of the proofs in characteristic zero 
in~\cite{MT}; we will only highlight the required modifications here.

\begin{prop}[Cf.~\cite{MT}, \S6.1 and \S6.3]\label{normalizer}
The map $\phi$ is a rational map from $\ucal$ into $N_G(\ucal)$. 
Furthermore, there is no compact subset $\mathcal{K}$ of $G$ such that ${\rm Im}(\phi)\subset \mathcal{K}\ucal.$
\end{prop}

\begin{proof}
Recall from~\eqref{e;chevalley2} that 
\[
N_G(\ucal)=\{g\in G:\rho(\ucal)\rho(g)\qpz=\rho(g)\qpz\}.
\] 
Thus, we need to show that for any $u_0\in \ucal$ and any $u\in\phi'^{-1}(\mathcal{M})$, we have 
\[
\rho(u_0)\rho(\phi(u))\qpz=\rho(\phi(u))\qpz;
\]
this suffices as $\phi'^{-1}(\mathcal{M})$ is a Zariski open, dense subset of $\ucal.$
 
Let $u\in\phi'^{-1}(\mathcal{M})$, then by~\eqref{e;quasi4} we have
\[
\phi(u)=\lim_{n\rightarrow\infty}\tau_n(u)g_n({w}_n(\tau_n(u)))^{-1}. 
\]
On the other hand, we have
$
\rho(u_0\tau_n(u)g_n)\qpz=\rho(\tau_n(\tau_{n}^{-1}(u_0)u)g_n)\qpz.
$

Note now that $\tau_n^{-1}(u_0)\rightarrow e$ as $n\rightarrow\infty.$ 
This, in view of the above discussion, implies that $\phi(u)\in N_H(\ucal)$ 
for all $u\in\phi'^{-1}(\mathcal{M}).$ The first claim follows.

To see the second assertion, note that $\phi=\rho_L^{-1}\circ\phi'.$
The claim thus follows since $\phi'$ is a non-constant polynomial map 
and $\rho_L$ is an isomorphism from $L$ onto a Zariski open, dense subset of the 
quasi affine variety $\rho(G)\qpz.$   
\end{proof}

In the sequel we will utilize a quasi-regular map, $\phi$, which is constructed 
using a sequence of elements $g_n\rightarrow e$ with the following property.

\begin{definition}[Cf.~\cite{MT}, Definition 6.6]\label{cstar}
A sequence $\{g_n\}$ is said to satisfy the condition $(*)$ with respect to $s$ 
if there exists a compact subset $\mathcal{K}$ of $G$ such that for all $n\in\bbn$ 
we have $s^{-b(n)} g_ns^{b(n)}\in \mathcal{K}.$
\end{definition}
This technical condition is used in the proof of the Basic Lemma.
It is also essential in the proof of Proposition~\ref{p;star}. 

We also recall the following

\begin{definition}\label{aver}
A sequence of measurable, non-null subsets $A_n\subset \ucal$ is called an 
{\em averaging net} for the action of $\ucal$ on $(X,\mu)$ 
if the following analog of the Birkhoff pointwise ergodic theorem holds. 
For any continuous, compactly supported function $f$ on $X$ and for almost all $x\in X$ one has
\begin{equation}\label{e;aver}
\lim_{n\rightarrow\infty}\frac{1}{\mu(A_n)}\int_{A_n}f(ux)d\theta(u)=
\int_Xf(h)d\mu_{y(x)}(h),
\end{equation}
where $\mu_{y(x)}$ denotes the $\ucal$-ergodic component 
corresponding to $x.$
\end{definition}

The proof of the following is standard. 

\begin{lem}[Cf.~\cite{MT}, \S7.2]\label{averlem}
Let $A\subset\ucal$ be open, relatively compact, and non-null. 
Let $A_n=\la_n(A).$ Then  $\{A_n\}$ is an averaging net for the action of $\ucal$ on $(X,\mu).$
\end{lem}

Ergodic theorems hold on full measure subsets of the space (with the exception of uniquely ergodic systems). 
The following is a uniform and quantitive version of {\em full measure} sets, and it is better adapted to limiting arguments.    

\begin{definition}\label{unifconv}
A compact subset ${\Omega}\subset X$ is said to be 
{\it a set of uniform convergence relative to $\{A_n\}$} if the following holds. 
For every $\vare>0$ and every continuous, compactly supported function 
$f$ on $X$ one can find a positive number $N(\vare,f)$ 
such that for every $x\in{\Omega}$ and $n>N(\vare,f)$ one has
\[
 \left|\frac{1}{\theta(A_n)}\int_{A_n}f(ux)d\theta(u)-\int_Xf(h)d\mu_{y(x)}(h)\right|<\vare.
\]
\end{definition}

The following is a consequence of Egoroff's Theorem and 
the second countability of the spaces under consideration, see~\cite[\S7.3]{MT}. 
\begin{lem}
For any $\vare>0$ one can find a measurable set ${\Omega}$ 
with $\mu({\Omega})>1-\vare$ which is  a set of uniform convergence relative to 
$\{A_n=\la_n(A)\}$ for every open, relatively compact, and non-null subset $A$ of $\ucal.$
\end{lem}

\subsection{}\label{sec:basic-lemma}
The following is the main application of the construction of the quasi-regular maps.
It provides us with the anticipated {\em extra invariance property}.

\begin{lem*}[Cf.~\cite{MT}, Basic Lemma, \S7.5]
\label{basic}
Let ${\Omega}$ be a set of uniform convergence relative to all averaging nets 
$\{A_n=\la_n(A)\}$ for all $A\subset\ucal$ which are open, relatively compact, and non-null. 
Let $\{x_n\}$ be a sequence of points in ${\Omega}$ with $x_n\rightarrow x\in{\Omega}.$ 
Let $\{g_n\}\subset G\setminus N_G(\ucal)$ be a sequence which satisfies condition $(*)$ with respect to $s.$ 
Assume further that $g_nx_n\in{\Omega}$ for every $n.$ 
Suppose $\phi$ is the $\ucal$-quasiregular map corresponding to $\{g_n\}$ constructed above. 
Then the ergodic component $\mu_{y(x)}$ is invariant under ${\rm{Im}}(\phi).$
\end{lem*}

\begin{proof}
The proof in~\cite[Basic Lemma]{MT} works the same here.
Indeed, the analysis simplifies in our situation as $U=U_0=\ucal$; we present a sketch here. 

Let the notation be as in the construction of $\phi_n$; in particular, $\tau_n=\la_{b(n)}$. 
Then, the condition $(*)$ allows one to write 
\[
\omega_n\circ\tau_n=\tau_n\circ\eta_n
\] 
where $\eta_n:\ucal\to\ucal$
is a rational map and locally a diffeomorphism. 
Furthermore, given $u_0\in\phi'^{-1}({\mcal}),$ the sequence $\{\eta_n\}$
uniformly converges to a diffeomorphism, $\eta$, on a neighborhood of $u_0.$ 

Let $A=\{u: \|\phi_n(u)f-\phi(u_0)f\|_\infty<\vare\}$ 
for $f\in C_c(X)$ and $\vare>0$ is small enough.
One then uses $\eta$ to compare the ergodic average 
\[
\frac{1}{\theta(A)}\int_{A} f\Bigl(\tau_n(u)g_nx\Bigr)\operatorname{d}\!\theta(u)
\]
with $\tfrac{1}{\theta(A)}\displaystyle\int_A f\Bigl(\phi_n(u)\omega_n(\tau_u(u))x\Bigr)\operatorname{d}\!\theta(u).$

Passing to the limit, we get that $\mu$ is $\phi(u_0)$-invariant, as was claimed.
\end{proof}

We finish this section with the following remark which will be used
in the proof of Theorem~\ref{t;U-measure-class}, see {\it Step 4} in the proof.

\begin{remark}\label{r;norm-inv}
The construction above assumed $\{g_n\}\not\in N_G(\ucal),$ 
however, we make the following observation. 
Let $\ucal, \mu$ and $\Omega$ be as in the Basic Lemma. 
Further, let us assume that $\mu$ is $\ucal$-ergodic.
Suppose\footnote{In general, $N_G(\ucal)\subset\zg\wpg$ can be arranged 
by embedding $G$ in some $\SL_n$ and choosing the element $s$ more carefully.
However, the more restrictive statement given here already suffices for our application.} 
\[
g\in N_G(\ucal)\cap\zg\wpg
\] 
is so that $gx\in\Omega$ for some $x\in\Omega.$ 
Then $\mu$ is invariant by $g.$

To see this, put $A=\Bfrak_0\cap \ucal$ and let $A_n=\la_n(A)$. 
For all $n\geq0$ and any continuous compactly supported function $f$ we have
\begin{align*}
\frac{1}{\theta(A_n)}\int_{A_n}f(ugx)d\theta(u)&=\frac{1}{\theta(A)}\int_{A}f(\la_n(u)gx)d\theta(u)\\
&=\frac{1}{\theta(A)}\int_{A}f(gg^{-1}\la_n(u)gx)d\theta(u)\\
&=\frac{1}{\theta(A)}\int_{A}f(g\la_n(g_n^{-1}ug_n)x)d\theta(u)\\
&=\frac{1}{\theta(B(n))}\int_{B(n)}f(g\la_n(u)x)d\theta(u),
\end{align*}
where $g_n=\lambda_n^{-1}(g)$ and $B(n)=g_n^{-1}Ag_n.$
In the last equality we used the fact that the Jacobian of the conjugation by $g_n$ is constant.  

Now let $g_0\in N_G(U)$ be so that $g_n\to g_0$ as $n\to\infty.$ 
Put $B=g_0^{-1}Ag_0,$ then $\theta(B(n)\bigtriangleup B)\to 0$ as $n\to\infty.$ 
Hence, for any $\vare>0$ and all large enough $n$, we have
\[
\left|\frac{1}{\theta(B(n))}\int_{B(n)}f(g\la_n(u)x)d\theta(u)-\frac{1}{\theta(B)}\int_{B}f(g\la_n(u)x)d\theta(u)\right|\leq \vare
\]
On the other hand since $x,gx\in\Omega,$ for all large enough $n$ we have
\[
\begin{array}{c}\left|\frac{1}{\theta(A_n)}\int_{A_n}f(ugx)d\theta(u)-\int_Xf(h)d\mu(h)\right|<\vare\\
\left|\frac{1}{\theta(B_n)}\int_{B_n}f(gux)d\theta(u)-\int_Xf(gh)d\mu(h)\right|<\vare.
\end{array}
\]
Putting all these together we get that
$
|\mu(f)-g\mu(f)|\leq 3\vare.
$
This implies the claim if we let $\vare\to0$.
\end{remark}


\section{Proof of Theorem~\ref{t;U-measure-class}}\label{sec;proof}
Let us recall the setup from the introduction. 
We fixed 
\begin{itemize}
\item a $\kt$-algebraic group $\bbg$, 
\item a closed subfield $k'\subset k_w$,
\item a $k'$-group $\bbh$, and 
\item a $k_w$-homomorphism $\iota:\bbh\times_{k'}k_w\to\bbg_{w}$. 
\end{itemize}
Therefore, if we replace $\bbg_{w}$ by $\bbg'_{w}:=\rcal_{k_w/k'}(\bbg_{w}),$ 
we get a $k'$-group, $\bbg'_w,$ such that $\bbg_w'(k')=\bbg_w(k_{w}).$ 
Furthermore, it follows from the universal property of Weil's restriction of scalars that 
$\bbh$ is a $k'$-subgroup of $\bbg_w'.$ 
Hence, we may and will assume that $k'=k_w$ and $\bbg_{w}=\bbg'_{w}.$ 
To simplify the notation, we will denote $k_{w}=k$ for the rest of this section.

We also have fixed a non central $k$-homomorphism, $\lambda:\bbg_m\to\bbh.$ 
Recall the one dimensional $k$-split tours $S\subset G$ defined using $\la$. 
Let $s=\iota(s')$, then $s\in S$ is an element from class $\mathcal A.$ 

\subsection{The subgroup $\ucal$}\label{sec:ucal}
As in the statement of Theorem~\ref{t;U-measure-class},
$\mu$ is a probability measure on $X=G/\Gamma$ which is $SU$-invariant and $U$-ergodic.
Define 
\be\label{eq:ucal-def}
\mbox{$\ucal\subset\wpg$ to be the maximal subgroup 
which leaves $\mu$ invariant.}
\ee  
Note that $U\subset\ucal.$
Since $\mu$ is $S$-invariant and $\wpg$ is normalized
by $S,$ the group $\ucal$ is a closed, in Hausdorff topology, subgroup of $\wpg$  
which is normalized by $S$. 

Therefore, Proposition~\ref{p:TorusInvariantUnipotent} implies the following. 
There exists some $q=p^n,$ depending on
the action of $\bbs$ on $\bbw^+_{\bbg_{w}}(s),$ 
such that $\ucal$ is the group of $k^q$-points of a connected, $k^q$-split, 
unipotent, $k^q$-subgroup of $\rest\Bigl(\bbw^+_{\bbg}(s)\Bigr).$   

We fix a $k^q$-homomorphism 
\[
\lambda':\bbg_m\to \rcal_{k/k^q}(\bbh)
\] 
and an element $s_0\in\la'(\bbg_b(k^q))$ so that $W^+_G(s)=W^+_G(s_0),$ see~\cite[App.\ C]{CGP}.

Replacing $\bbg_{w}$ with $\rest(\bbg_{w}),$
which we continue to denote by $\bbg_w,$ 
we have $\ucal$ is an algebraic subgroup of $G_{w}.$

Also replace $\bbs$ by $\lambda'(\bbg_m),$ $S$ by $\rcal_{k/k^q}(\iota)(\lambda'(\bbg_m)(k^q))$
and $s$ by $s_0.$
Finally, we replace $k$ by $k^q$ in $\kt$ and continue to denote this by $\kt.$

In particular, we have the following 

\begin{itemize}
\item The group $\ucal$ is the set of $\kt$-points of a connected, $\kt$-split, unipotent, $\kt$-subgroup of $\wpg.$
\item The group $S$ is a $k$-split one dimensional $k$-torus, $s\in S$, and $\ucal\subset\wpg$ is normalized by $S.$ 
\end{itemize}

\subsection{The subgroup $\fcal(s)$}\label{sec:fs} 
Following~\cite{MT}, we define
\begin{equation}\label{e;counterpart1}
\fcal(s)=\{g\in G:\h\ucal g\subset{\overline{\wmg\zg\ucal}}^z\};
\end{equation}
as indicated, the closure is the Zariski closure. 
Since $\wmg\zg$ is a subgroup of $G$ the above can be written as 
\begin{equation}\label{e;counterpart2}
\{g\in G:\h{\overline{\wmg\zg\ucal}}^z g\subset{\overline{\wmg\zg\ucal}}^z\}.
\end{equation}
Thus the inclusion in~\eqref{e;counterpart2} may be replaced by equality. 
This implies 
\[
 \mbox{$\fcal(s)$ is a $k$-closed subgroup of $G.$}
\]
Note that $S\subset\fcal(s)$ and $\fcal(s)\cap\wpg=\ucal.$ 

Put $\ucal^-=\fcal(s)\cap\wmg.$ 
This a $k$-closed subgroup of $\wmg$ which is normalized by $S.$ 

\begin{lem}\label{lem:ucal-cross}
The group $\ucal^-$ is the $k$-points of a connected, unipotent, 
$k$-subgroup of $W^-_G(s).$ Moreover, there is a regular cross-section, $\vcal^-$, for $\ucal^-$ in $W^-_G(s)$
and $\vcal^-$ is invariant under conjugation by $s$. 
\end{lem}

\begin{proof} 
The second claim follows from the first claim
and~\cite[Prop.~9.13]{BS}.

To see the first claim, note that by Lemma~\ref{l;group-scheme},
there is a smooth group scheme $\mathbb B$ defined over $k$ so that $\mathbb B(k)=\ucal^-.$
Hence, $\mathbb B\subset \mathbb W^-_{\mathbb G}(s)$ is a unipotent group which is normalized 
by $s.$ Since $s$ contracts every element of $\mathbb B$ to the identity, 
we get that $\mathbb B$ is connected as was claimed.  
\end{proof}

Similarly, let $\vcal$ be a regular cross-section for $\ucal$ in $\wpg$ which is invariant under conjugation by $s$.

\subsection{Structure of $\mu$ along contacting leaves of $s$}
In this section we will use the maximality of $\ucal$
and the Basic Lemma to show that $\mu$ has a rather special 
structure along $\wmg.$ The main result is Proposition~\ref{p;contracting-leaf1}.
We first need some more notation. Put
\begin{equation}\label{e;w-zw+}
\mathcal{D}=\wmg\zg\wpg=\ucal^-\vcal^-Z_G(s)\vcal\ucal.
\end{equation} 
Then, $\mathcal{D}$ is a Zariski open dense subset of $G$ containing $e,$ see \S\ref{sec;alg-group}.
Moreover, for any $g\in\mathcal{D}$ we have a unique decomposition 
\begin{equation}\label{decomposition} 
g=w^-(g)z(g)w^+(g)=u^-(g)v^-(g)z(g)v(g)u(g)
\end{equation}
where $u^-(g)\in\ucal^-,\h v^-(g)\in\vcal^-,
\h z(g)\in\zg,\h u(g)\in\ucal,\h v(g)\in\vcal,\h w^-(g)=u^-(g)v^-(g),$ and $w^+(g)=v(g)u(g).$

Note that for every $w^{\pm}\in \wpmg$ we have 
\begin{equation}\label{e;ell-pm}
\ell^{\pm}(s^mw^{\pm}s^{-m})=\ell^{\pm}(w^{\pm}(g))\pm m.
\end{equation}
We need the following 

\begin{prop}
\label{p;star}
Suppose $\{g_n\}$ is a sequence converging to $e,$ and  
let $s$ and $\ucal$ be as above. Suppose one of the following holds 
\begin{enumerate}
\item the sequence $\ell^-(v^-(g_n))-\ell^-(u^-(g_n))$ is bounded from below, or 
\item $\{g_n\}\subset\zg\wpg\setminus N_G(\ucal).$
\end{enumerate}
Then, $\{g_n\}$ satisfies the condition $(*).$ 
Furthermore, if we let $\phi$ be the quasi-regular map constructed using $\ucal$ and this sequence $\{g_n\},$ 
then ${\rm Im}(\phi)\subset {\wpg}.$
\end{prop}

\begin{proof}
The fact that the conclusion holds under condition (1) is proved in
~\cite[Prop.~6.7]{MT}. We show (2) also implies the conclusion. 

Under assumption (2), we have 
$s^{-b(n)}g_ns^{b(n)}\to e,$ hence, $\{g_n\}$ satisfies the condition $(*)$. We now 
use an argument similar to
~\cite[Prop.~6.7]{MT} to show ${\rm Im}(\phi)\subset\wpg$ when condition (2)
above holds.
By~\eqref{e;quasi3} we have
\[
\phi(u)=\lim_{n\to\infty}\phi_n(u)=\lim_{n\to\infty}\tau_n(u)g_n{\omega}_n(\tau_n(u))^{-1}\mbox{ for all $u\in\phi'^{-1}(\mcal).$ }
\] 
It follows from the choice of $b(n)$
that 
\[
 \mbox{$\{s^{-b(n)}\bullet s^{b(n)}\}$ are bounded in $\ucal$
for $\bullet=\tau_n(u)$ and ${\omega}_n(\tau_n(u))^{-1}.$}
\]
In view of this and the condition $(*)$, we get the following. 
After possibly passing to a subsequence, we have  
\[
\lim_{n\to\infty}s^{-b(n)}\phi_n(u)s^{b(n)}=
\lim_{n\to\infty} s^{-b(n)}\tau_n(u)g_n{\omega}_n(\tau_n(u))^{-1}s^{b(n)}\in\ucal.
\]
This implies that $\phi(u)\in\wpg$ for all $u\in\phi'^{-1}(\mcal).$ 
Together with the fact that $\phi'^{-1}(\mcal)$ is Zariski dense in $\ucal$, this finishes the proof.
\end{proof}

The following is an important consequence of the above proposition 
and the construction of quasi-regular maps in \S\ref{sec;quasi}. 
It describes the local structure of the set of uniform convergence. 
Our formulation here is taken from~\cite{MT};
let us remark that obtaining such description is also essential in~\cite{Rat3}.  

\begin{prop}
\label{p;contracting-leaf1}
For every $\vare>0$, there exists a compact subset ${\Omega}_{\vare}\subset X$ 
with the following properties. 
\begin{enumerate}
\item $\mu({\Omega}_{\vare})>1-\vare$, and
\item if $\{g_n\}\subset G\setminus N_G(U^+(s))$ is a sequence so that $g_n\rightarrow e$ and
\[
 g_n{\Omega}_{\vare}\cap{\Omega}_{\vare}\neq\emptyset \mbox{ for every $n,$}
\]
then the sequence $\{\ell^-(v^-(g_n))-\ell^-(u^-(g_n))\}$ tends to $-\infty.$
\end{enumerate}
\end{prop}

\begin{proof}
First note that $U\subset\ucal,$ therefore,
$\mu$ is $\ucal$-ergodic and invariant. 
Let $\vare>0$ be given. Let ${\Omega}_\vare$ be a set of uniform convergence 
for the action of $\ucal,$ in the sense of Definition~\ref{unifconv}, with $\mu({\Omega}_\vare)>1-\vare.$ 
We will show that ${\Omega}_\vare$ satisfies (2) in the proposition as well. 

Assume contrary to our claim that there is a sequence $\{g_n\}$ so that (2) fails.
Passing to a subsequence we may assume the sequence 
\[
\Bigl\{\ell^-(v^-(g_n))-\ell^-(u^-(g_n))\Bigr\}
\] 
is bounded from below. 
Therefore, Proposition~\ref{p;star} guarantees that $\{g_n\}$ satisfies condition $(*).$ 

Now construct the quasi-regular map $\phi$ corresponding to $\{g_n\}$ as in \S\ref{sec;quasi}.
Then,
\begin{itemize}
\item in view of Proposition~\ref{normalizer}, the image of $\phi$ is not contained in $\mathcal{K}\ucal$ 
for any bounded subset $\mathcal{K}\subset G$. 
\item It follows from the Basic Lemma that $\mu$ is invariant under ${\rm{Im}}(\phi).$ 
\item By Proposition~\ref{p;star}, we have ${\rm{Im}}(\phi)\subset {\wpg}.$ 
\end{itemize}
Therefore, $\mu$ is invariant under $\langle \ucal,{\rm{Im}(\phi)}\rangle$ 
which is contained in $W^+_G(s)$ and strictly contains $\ucal.$ 
This contradicts the maximality of $\ucal$ and finishes the proof.   
\end{proof}

We will use this proposition in the following form.

\begin{cor}
\label{c;contarcting-leaf2}
There exists a subset ${\Omega}\subset X$ with $\mu({\Omega})=1$ 
such that 
\[
{\wmg}x\cap{\Omega}\subset \ucal^-x
\] 
for every $x\in{\Omega}.$
\end{cor}

\begin{proof}
The proof follows the same lines as the proof of ~\cite[Cor.~8.4]{MT} 
using Proposition~\ref{p;contracting-leaf1}.
We recall the proof here for the convenience of the reader. 

First let us note that by Mautner's phenomenon, every $s$-ergodic component of $\mu$
is $\ucal$-invariant, thus, $\mu$ is $s$-ergodic. 
For any $\vare>0$, let ${\Omega}_\vare$ be as in Proposition~\ref{p;contracting-leaf1}. 
Let ${\Omega}'_\vare\subset\Omega_\vare$ be a compact subset with $\mu(\Omega_\vare')>1-2\vare$ 
so that the Birkhoff ergodic theorem for the action of $s$ and $\chi_{\Omega_\vare}$ 
holds for every $x\in{\Omega}'_\vare$. 

Suppose $x$ and $y=w^-x$ are in ${\Omega}'_\vare$
and assume that $w^-\notin\ucal^-.$ Let $n_i\to\infty$ be a subsequence so that both $s^{n_i}x\in{\Omega}_\vare$
and $s^{n_i}y\in{\Omega}_\vare,$ such sequence exists by Birkhoff ergdoic theorem.

Let $x_i=s^{n_i}x$ and 
\[
 y_i=s^{n_i}y=s^{n_i}ws^{-n_i}s^{n_i}x=w_ix_i
\]
where $w_i=s^{n_i}ws^{-n_i}.$
Our assumption on $w$ and~\eqref{e;ell-pm} imply
that 
\[
\{\ell^-(v^-(w_i))-\ell^-(u^-(w_i))\}
\] 
is bounded from below which contradicts Proposition~\ref{p;contracting-leaf1}. 

The corollary now follows if we apply the above argument to a sequence $\vare_n\to0$
and let $\Omega=\cup_n\Omega'_{\vare_n}.$
\end{proof}

\subsection{A lemma on finite dimensional representations}
We need certain properties of the subgroup $\fcal(s)$ which was defined 
in~\eqref{e;counterpart1}.
These will be used when we apply Theorem~\ref{entropy}
in the proof of Theorem~\ref{t;U-measure-class}. 
The main property needed is Lemma~\ref{l;exp-geq-cont} below which 
is a consequence of Lemma~\ref{l;epimorphic}.
It is worth mentioning that the latter is closely related to the notion of an {\it epimorphic group}
which was introduced by A.~Borel.

Retain the notation from \S\ref{sec;intro}. 
Recall also from the reductions in the beginning of this section that $k=k'.$ 
Put $H^+=\langle\wph,\wmh\rangle.$ 
Since $\bbh'$ is $k''$-almost simple and $k''$-isotropic we have
\[
H^+=\tilde\iota\Bigl(\rcal_{k''/k}(\widetilde{\bbh'})(k)\Bigr)
\]
where $\widetilde{\bbh'}\to \bbh'$ denotes the simply connected cover of $\bbh',$  
and $\tilde\iota$ is $\iota$ precomposed with the covering map, see~\cite[Prop.~1.5.4 and Thm.~2.3.1]{Mar5}.

\begin{lem}
\label{l;epimorphic}
Let $(\rho,\Phi)$ be a finite dimensional representation of $G_w$
defined over $k$ and let $\|\;\|$ denote a norm on $\Phi.$
Let $\qpz\in\Phi.$ 
\begin{enumerate}
\item If $\langle U,s\rangle\subset\{g\in G_w: \rho(g)\qpz\in k\cdot\qpz\},$ 
then $\|\rho(s^n)\qpz\|\geq \|\qpz\|$ for all $n\geq 1.$
\item If $\rho(s)\qpz=\qpz$ and $\rho(U)\qpz=\qpz,$ then $\rho(H^+)\qpz=\qpz.$
\end{enumerate}
\end{lem}

\begin{proof} 
The argument is similar to the one given in~\cite[Lemma 5.2]{Sh}.

First note that since $U$ is a unipotent subgroup of $G$
our assumption in (1) implies that $\rho(U)\qpz=\qpz,$
and that $\rho(s)\qpz=\chi(s)\qpz$ for some $k$-character $\chi.$  

Assume the contrary to (1). Then since $s$ acts by a character, we get that 
\be\label{eq:lim-q-0}
\lim_{n\to\infty}\rho(s^n)\qpz=0.
\ee
Let $\Phi^-$ denote the subspace of $\Phi$
corresponding to the negative weights of the action of $\rho(s)$ on $\Phi$. 
Therefore,~\eqref{eq:lim-q-0} is to say $\qpz\in\Phi^-.$ 

We claim~\eqref{eq:lim-q-0} implies the following
\be\label{eq:H-+-q}
\rho( H^+)\qpz\subset\Phi^-.
\ee 
Let us assume~\eqref{eq:H-+-q} and conclude the proof first.
Indeed~\eqref{eq:H-+-q} implies that
\[
\Psi:=\text{$k$-linear span of }\{\rho( H^+)\qpz\}\subset\Phi^-. 
\]
Let $\varrho:H^+\to\GL(\Psi)$ be the corresponding representation. 
Now by Lemma~\ref{lem:H+-co-torsion} there is some $n_0$ so that $s^{n_0}\in H^+.$
Hence 
\[
|{\rm det}(\varrho(s^{n_0}))|\neq1.
\]
This contradicts the fact that $H^+$ is a generated by 
$k$-unipotent subgroups and finishes the proof of~(1).

We now turn to the proof of~\eqref{eq:H-+-q}. 
Recall that $\qpz$ is fixed by $U,$ therefore,~\eqref{eq:H-+-q} 
follows if we show that 
\[
\rho(\wmh)\subset\Phi^-.
\] 
Let $w\in\wmh$ be arbitrary. Then $s^nws^{-n}\to e$ as $n\to\infty.$ 
Using this and~\eqref{eq:lim-q-0} we get that
\[
\lim_{n\to\infty}\rho(s^n)\rho(w)\qpz=\lim_{n\to\infty}\rho(s^nws^{-n})\rho(s^n)\qpz= 0.
\]
Hence, $\rho(w)\qpz\in\Phi^-$ for all $w\in\wmh$ as we wanted to show.

We now prove (2). The proof is similar to the above.
Decompose $\Phi$ according to the weights of the element $s\in S.$
Hence 
\[
\Phi=\Phi^-+\Phi^0+\Phi^+.
\] 
 
We claim that 
\be\label{eq:wm-z}
\rho(W^-_{H}(s)Z_{H}(s))\qpz\subset\Phi^-+\Phi^0
\ee
To see this, let $wz\in W^-_{H}(s)Z_{H}(s)$ be arbitrary. Then $s^nws^{-n}\to e$ as $n\to\infty$
and $sz=zs.$ Since $\rho(s)\qpz=\qpz$ we get that
\[
\lim_{n\to\infty}\rho(s^n)\rho(wz)\qpz=\lim_{n\to\infty}\rho(s^nws^{-n})\rho(z)\rho(s^n)\qpz=\rho(z)\qpz.
\]
Hence, $\rho(wz)\qpz\in\Phi^-+\Phi^0$ for all $wz\in W^-_{H}(s)Z_{H}(s)$ as we wanted to show. 

Our assumption that $\rho(U)\qpz=\qpz$ together with~\eqref{eq:wm-z} now implies that
\[
\rho(W^-_{H}(s)Z_{H}(s)W^+_{H}(s))\qpz\subset\Phi^-+\Phi^0.
\] 
This is to say 
$\rho\circ\iota\Bigl(\bbw^-_{\bbh}(s)(k)\bbz_{\bbh}(s)(k)\bbw^+_{\bbh}(s)(k)\Bigr)\qpz\subset\Phi^-+\Phi^0.$  
In view of~\eqref{opposite-horo} we thus get
\[
\rho(H^+)\qpz\subset\rho(H)\qpz=\rho\Bigl(\iota(\bbh(k))\Bigr)\qpz\subset\Phi^-+\Phi^0. 
\]

As above, define $\Psi$ to be the $k$-span of $\{\rho(H)\qpz\}.$ Note that $\Psi\subset\Phi^-+\Phi^0.$ 
Let $(\varrho,\Psi)$ denote the corresponding representation of $H$ on $\Psi.$  
Let $n_0$ be so that $s^{n_0}\in H^+.$ Since $H^+$ is generated by $k_w$-unipotent 
subgroups we get that
\[
|{\rm det}(\varrho(s^{n_0}))|=1,
\]
which implies that $\Psi\subset\Phi^0.$

Let now $\ppz\in\Psi$ be any vector. 
For any compact subset $B\subset U$ there is a compact subset $B'\subset \Psi$
so that $\varrho(B)\ppz\subset B'.$ Since $\Psi\subset\Phi^0,$ we get
\be\label{eq:exp-U}
\varrho(s^{nn_0}Bs^{-nn_0})\ppz=\varrho(s^{nn_0})\varrho(B)\ppz\subset B'.
\ee
Letting $n\to\infty$ we get from~\eqref{eq:exp-U} that $\rho(U)\ppz\subset B'$ 
is contained in a compact subset of $\Psi$.
Note, however, that $U$ is a $k$-split, unipotent, $k$-subgroup, therefore 
\[
\varrho(U)\ppz=\ppz.
\]
Hence, $U$ is in the kernel of $\varrho.$ 
Since the kernel is a normal subgroup of $H^+$
we get from Lemma~\ref{lem:H+-co-torsion} that 
$H^+\subset\ker(\varrho)$ which implies (2).
\end{proof}

\begin{lem}\label{lem:H+-co-torsion}
The only normal subgroup of $H^+$ which contains $U$ is $H^+.$
Moreover, $H/H^+$ is a compact and torsion group. 
\end{lem}

\begin{proof}
Note that $\bbh(k)=\bbh'(k'')$ and $s\in\iota(\bbh(k)).$
Therefore, 
\[
H^+=\iota\Bigl(\bbh'(k'')^+\Bigr).
\]
Since $\bbh'$ is $k''$-almost simple and $k''$-isotropic, 
it follows from~\cite[Ch.~I, Prop.~1.5.4, Thms.~1.5.6, and Thm.~2.3.1]{Mar5}
that the claims in the lemma hold for $\bbh'(k'')^+.$
Hence, they hold for $H^+.$
\end{proof}

Let $\bbf$ denote the Zariski closure of the $k$-closed group $\fcal(s).$ 
By Lemma~\ref{l;group-scheme} we have 
$\bbf$ is a $k$-algebraic subgroup of $\bbg$ and $\fcal(s)=\bbf(k).$
It is worth mentioning that $\bbf$ is not necessarily connected we let $F=\bbf^\circ(k)$. 
Using the adjoint action of $s$ we have
\[
\Lie(\bbf)=\Lie(\ucal^-)\oplus\Lie(\bbf\cap\zg)\oplus\Lie(\ucal)\subset\Lie(\bbg).
\]
Recall from the introduction that the product map
\be\label{e;f-iwasawa}
\ucal^-\times (F\cap\zg)\times\ucal\to F
\ee
is a diffeomorphism onto a Zariski open dense subset which contains the identity.

Let us fix a norm $\|\;\|$ on $\Lie(\bbg).$
Put $\Phi=\wedge^{\dim\bbf}\Lie(\bbg),$ $\rho=\wedge^{\dim\bbf}\Ad$, and
let $\qpz\in k\cdot\wedge^{\dim\bbf}\Lie(\bbf)$ be a nonzero vector. Then
\be\label{eq:fs-qpz}
\fcal(s)\subset\{g\in G: \rho(g)\qpz\in k\cdot\qpz\}.
\ee
Moreover, since $\ucal$ is a unipotent subgroup $\rho(\ucal)q=q.$
Recall now that $U\subset\ucal,$ therefore, 
\be\label{eq:Uq=q}
\rho(U)\qpz=\qpz.
\ee

\begin{lem}\label{l;exp-geq-cont}
We have
\begin{equation}\label{e;exp=cont}
\alpha(s,\ucal)\geq\alpha(s^{-1},\ucal^-)
\end{equation} 
where the function $\alpha$ denotes the modulus of the conjugation action, 
see~\S\ref{sec;modulus}. 
\end{lem}

\begin{proof}
We first note that in view of relations between the Haar measure and 
algebraic form of top degree, see~\cite[10.1.6]{Bour-Var} and~\cite[Thm.~2.4]{Oes}, 
the claim in the lemma is equivalent to the fact that
\be\label{eq:det-s}
|{\rm det}(\Ad (s))|_{\Lie(\ucal)}|_{w}\geq |{\rm det}(\Ad (s^{-1}))|_{\Lie(\ucal^-)}|_{w}.
\ee
In view of the definition of $\rho$ and $\qpz,$~\eqref{eq:det-s} follows if we show $\|\rho( s)\qpz\|\geq\|\qpz\|.$
The latter holds thanks to Lemma~\ref{l;epimorphic}(1) in view of~\eqref{eq:fs-qpz} and~\eqref{eq:Uq=q}.
\end{proof}

\subsection{Entropy argument and the conclusion of the proof}\label{sec:entropy-proof}
The following theorem is proved in~\cite{MT}; it serves as one of the main
ingredients in the proof of the measure classification theorem in~\cite{MT}. 
 
Given an element $s$ from class $\mathcal{A}$ 
which acts ergodically on a probability measure space $(X,\sigma)$, 
let $\mathsf h_\sigma(s)$ denote the measure theoretic entropy of $s.$  

\begin{thm}[Cf.~\cite{MT}, Theorem 9.7]\label{entropy}
Assume $s$ is an element from class $\mathcal{A}$ which acts ergodically on a measure space $(X,\sigma).$ 
Let $V$ be an algebraic subgroup of ${\wmg}$ which is normalized by $s$ and put $\alpha=\alpha(s^{-1},V).$
\begin{enumerate}
\item If $\sigma$ is $V$-invariant, then $\mathsf h_\sigma(s)\geq\log\alpha.$
\item Assume that there exists a subset ${\Omega}\subset X$ with $\sigma({\Omega})=1$ 
such that for every $x\in{\Omega}$ we have ${\wmg}x\cap{\Omega}\subset Vx.$ 
Then $\mathsf h_\sigma(s)\leq\log(\alpha)$ and the equality holds if and only if $\sigma$ is $V$-invariant.
\end{enumerate}
\end{thm}

\begin{proof}[Proof of Theorem~\ref{t;U-measure-class}]
Let $\mu$ be as in the statement of Theorem~\ref{t;U-measure-class}. 
First note that $\mu$ is $s$-ergodic. Indeed by Mautner phenomenon 
any $s$-ergodic component of $\mu$ is $U$-invariant. 
Therefore, $s$-ergodicity of $\mu$ follows its $U$-ergodicity.  

Let $\ucal$ be as in~\eqref{eq:ucal-def}, i.e.~$\ucal$ is the maximal subgroup of $\wpg$ which leaves $\mu$ invariant. 
We complete the proof in some steps. 

\vspace{1mm}
{\it Step 1.} $\mu$ is invariant under $\ucal^-.$

By Corollary~\ref{p;contracting-leaf1}, there exists a full measure subset ${\Omega}\subset X$ such that 
\[
{\wmg}x\cap{\Omega}\subset \ucal^-x \text{ for every }x\in{\Omega}.
\] 
Recall that $\mu$ is $s$-ergodic. 
Applying Theorem~\ref{entropy}, we get that
\begin{equation}\label{e;entropy1}
\log\alpha(s,\ucal)\leq \mathsf h_\mu(s)=\mathsf h_{\mu}(s^{-1})\leq\log\alpha(s^{-1},\ucal^-).
\end{equation}
Note, however, that by Lemma~\ref{l;exp-geq-cont} we have  
\[
\alpha(s,\ucal)\geq\alpha(s^{-1},\ucal^-).
\]
Therefore, the equality must hold in~\eqref{e;entropy1}.

Now another application of Theorem~\ref{entropy}(2) 
implies that $\mu$ is invariant under $\ucal^-.$

{\it Step 2.} Reduction to Zariski dense measures.

We now apply Lemma~\ref{zd-measure1} with $B=\langle \ucal^-,S,\ucal\rangle$ and $M=G.$ 
Hence, we get a connected $\kt$-subgroup $\bbg'\subset \bbg$ and a point $g\Gamma=x\in X$
such that $B\subset G'$ and $\mu(G'x)=1$ where $G':=\bbg'(\kt).$ 
Moreover, for every proper $\kt$-closed subset $D\subsetneq G'$ 
we have $\mu(\pi(D))=0,$ and $G'\cap g\Gamma g^{-1}$ is Zariski dense in $\bbg'.$  

Abusing the notation we let 
$\vcal$ (resp.~$\vcal^-$) denote an $S$-invaraint cross section for $\ucal$ (resp.~$\ucal^-$)
in $\wpgp$ (resp.\ $\wmgp$).

{\it Step 3.} $\mu$ is invariant under ${\wmgp}.$

We will show $\ucal^-={\wmgp}$ which implies the claim in view of {\em Step 1.} 

Assume the contrary. Then $\vcal^-\neq\{e\}.$ 
The definition of $\ucal^-,$ see~\eqref{e;counterpart1}, implies that 
\[
 \vcal^-\zgp{\wpgp}\not\subset N_{G'}(\ucal).
\]
In particular, 
\[
D':=\Bigl(\zgp{\wpgp}\cup N_{G'}(\ucal)\Bigr)\cap \vcal^-\zgp{\wpgp}
\] 
is a proper $\kt$-closed subset of $\vcal^-\zgp{\wpgp}.$

This together with {\em Step 1} and {\em Step 2} implies that conditions in Lemma~\ref{zd-measure2} 
are satisfied with $M=G',$ $B=\ucal^-,$ $L=\vcal^-\zgp{\wpgp},$ and $D=D'.$
Therefore, we get the following from the conclusion of that lemma. 
If $0<\vare<1$ is small enough and ${\Omega}_\vare$ 
is a measurable set with $\mu({\Omega}_\vare)>1-\vare,$ 
then there exists a sequence $\{g_n\}$ converging to $e$ such that
\begin{equation}\label{w-displacement}
\{g_n\}\subset \vcal^-\zgp{\wpgp}\setminus \Bigl(\zgp{\wpgp}\cup N_{G'}(\ucal)\Bigr),
\end{equation}
and $g_n{\Omega}_\vare \cap {\Omega}_\vare\neq\emptyset$ for all $n.$ 

In particular, we have $\ell^-(v^-(g_n))>-\infty$ and $\ell^-(u^-(g_n))=-\infty.$ 
This contradicts Proposition~\ref{p;contracting-leaf1}
and shows that $\ucal^-={\wmgp}$ as was claimed.

{\it Step 4.}  $\mu$ is invariant under ${\wpgp}.$

We will show $\ucal=\wpgp$ which implies the claim. 

Assume the contrary.
Then $\zgp\ucal$ is a proper subvariety of $\zgp\wpgp.$
Apply Lemma~\ref{zd-measure2} with $M=G',$ $B=\wmgp,$ $L=\zgp\wpgp$,
$D=\zgp\ucal,$ and ${\Omega}_\vare$ as in {\em Step 3.}
Therefore, we find 
\[
\mbox{$\{g_n\}\subset \zgp\wpgp\setminus \zgp\ucal$}
\]
so that $g_n{\Omega}_\vare\cap{\Omega}_\vare\neq\emptyset$ and $g_n\to e.$ 

We consider two cases.

{\it Case 1.} 
Suppose there is a subsequence $\{g_{n_i}\}$ such that $g_{n_i}\notin N_{G'}(\ucal)$
for all $i.$ Abusing the notation, we denote this subsequence by $\{g_n\}.$  
Construct the map $\phi$ using this sequence $\{g_n\}.$ 
Then, $\mu$ is invariant by $\langle\ucal,{\rm Im}(\phi)\rangle.$ 
On the other hand by Proposition~\ref{p;star}(2) we have ${\rm Im}(\phi)\subset\wpgp\setminus \ucal.$
This contradicts the maximality of $\ucal$ and finishes the proof.

{\it Case 2.} Suppose there exists some $n_0$ so that $g_n\in N_{G'}(\ucal)$
for all $n\geq n_0.$ Taking $n\geq n_0$, we assume that $g_n\in N_{G'}(\ucal)$ for all $n.$ 
Now by Remark~\ref{r;norm-inv}, $\mu$ is invariant under $g_n$ for all $n.$
Write 
\[
g_n=z(g_n)v(g_n)u(g_n)\in \zgp\vcal\ucal.
\] 
Since $\mu$ is $\ucal$-invariant, the above implies 
that $\mu$ is invariant under $z(g_n)v(g_n).$ 
Moreover, in view of the choice of $g_n$ we have
$v(g_n)\neq e.$ 
Recall also that $\mu$ is $s$-invariant. Therefore, $\mu$ is invariant under 
\[
\mbox{$s^\ell z(g_n)v(g_n) s^{-\ell}=z(g_n)s^\ell v(g_n) s^{-\ell}.$}
\]
for all $\ell\in\bbz.$
For each $n$ choose $\ell_n\in\bbz$ so that 
\[
v_n=s^{\ell_n}v(g_n)s^{-\ell_n}\in(\vcal\cap\Bfrak_0)\setminus\Bfrak_{-1}; 
\]
such $\ell_n$ exists since $v(g_n)\neq e$ and the cross section $\vcal$ 
is $S$-invariant. 
Now passing to a subsequence,
we get that $z(g_n)v_n\to v\in\vcal$ and $v\neq e.$ Therefore, $\mu$ is invariant under $v.$  
This again contradicts maximality of $\ucal.$ 

{\it Step 5.} Conclusion of the proof.

So far we have proved that $\mu(G'x)=1$ and $\mu$
is invariant and ergodic under $G'':=\langle {\wmgp},{\wpgp}\rangle.$

Hence, $\mu$ is a probability measure on $G'/G'\cap g\Gamma g^{-1}$ 
which is invariant and ergodic under $G''.$
By Lemma~\ref{lem:w-pm-normal} we have that $G''$ is a normal and unimodular subgroup of $G',$
see ~\eqref{eq:M-+-la} and~\eqref{eq:def-W-pm-s}.
This and Lemma~\ref{normal-unimodular} now imply Theorem~\ref{t;U-measure-class}.
\end{proof}

We now give a refinement of Theorem~\ref{t;U-measure-class}.

\begin{thm}\label{algeb-measure}
Let the notation and the assumptions be as in Theorem~\ref{t;U-measure-class}. 
Then, there exist
\begin{enumerate}
\item $\lt=\prod_{{v}\in\tcal}l_{v}\subset \kt$ where $l_{v}=k_{v}$ if $v\neq w$ and $l_{w}=(k')^q$ for some $q=p^n,$ moreover, $q$ depends only on the weights of the action of $S$ by conjugation, 
\item a connected $\lt$-subgroup $\bbf$ of $\rcal_{\kt/\lt}(\bbg)$ so that $\bbf(\lt)\cap \Gamma$
is Zariski dense in $\bbf,$ 
\item a point $x=g_0\Gamma\in X,$ 
\end{enumerate} 
such that $\mu$ is the $\Sigma$-invariant probability Haar measure on the closed orbit $\Sigma x$ 
with 
\[
\Sigma=g_0\overline{F^+(\la)(\bbf(\lt)\cap\Gamma)}g_0^{-1} 
\]
where  
\begin{itemize}
\item the closure is with respect to the Hausdorff topology, and 
\item $F^+(\la)$ is defined in~\eqref{eq:M-+-la} for a non central $\lt$-homomorphism $\la:\bbg_m\to\bbf.$ 
\end{itemize} 
Moreover, $g_0{\mathbb F}(\lt)\Gamma$
is the smallest set of the form $\bbm(\lt)\Gamma$ where $\bbm$ is an $\lt$-subvariety
so that $\mu(\bbm(\lt)\Gamma/\Gamma)>0.$
\end{thm}

\begin{proof}
Indeed the above assertions are proved in the course of the proof of Theorem~\ref{t;U-measure-class}. 
We give a more detailed discussion here for the sake of completeness.

In view of the discussion in the beginning of \S\ref{sec;proof}, we have the following.
\begin{itemize}
\item[(a)] There is some $\lt\subset\kt$ as in (1) so that the group 
$\ucal$ which is defined in~\eqref{eq:ucal-def}
is an $\lt$-split, unipotent, $\lt$-subgroup of $\mathcal R_{\kt/\lt}(\bbg).$
\end{itemize}
We thus replace $\bbg$ with $\mathcal R_{\kt/\lt}(\bbg)$ and have $SU\subset G,$ see \S\ref{sec:ucal}.
By {\em Step 2} in the proof of Theorem~\ref{t;U-measure-class} we have the following.
\begin{itemize}
\item[(b)] There is a connected $\lt$-subgroup, $\mathbb E$, of minimal dimension
such that 
\[
S\ucal\subset E:=\mathbb E(\lt)
\] 
and a point $x=g_0\Gamma\in X$ with the following properties. 
$\mu$ is a probability measure on $E/E\cap g_0\Gamma g_0^{-1},$ moreover, $E\cap g_0\Gamma g_0^{-1}$ is Zariski dense in $\mathbb E.$
\end{itemize}
By {\em Step 5} in the proof of Theorem~\ref{t;U-measure-class} we have the following.
\begin{itemize}
\item[(c)] $\mu$ is the $\Sigma$-ergodic invariant measure on the closed orbit
the closed orbit $\Sigma x,$ where 
\[
\Sigma=\overline{E'(E\cap g_0\Gamma g_0^{-1})}.
\]
with $E':=\langle W^+_{E}(s),W^-_{E}(s)\rangle.$
\end{itemize}
  
Note that in view of (b) above $g_0^{-1}Eg_0\cap\Gamma$ is Zariski dense in $g_0^{-1}\mathbb Eg_0.$ 
Therefore, 
\[
\bbf:=g_0^{-1}\mathbb Eg_0\subset\rcal_{\kt/\lt}(\bbg),
\]
$\lt$ as in (a), and $g_0$ as in (b) satisfy the claims in the theorem.

The final claim follows from Lemma~\ref{zd-measure1} and {\em Step 2} in the proof of Theorem~\ref{t;U-measure-class}.
\end{proof}


\begin{cor}\label{c;alg-measure-S-arith}
The conclusion of Theorem~\ref{algeb-measure} holds in the setting of Theorem~\ref{c;H-measureclass}.
\end{cor}

We conclude this section with the following lemma.

\begin{lem}\label{lem:H+-inv}
Let the notation be as in Theorem~\ref{algeb-measure} and its proof.
Assume further that there is an $\lt$-representation $(\rho,\Phi)$
and a vector $\qpz\in\Phi$ so that 
\be\label{eq:E-observable}
\mathbb E=\{g\in \rcal_{\kt/\lt}(\bbg):\rho(g)\qpz=\qpz\}.
\ee
Then, $\mu$ is invariant under $H^+.$ 
\end{lem}

\begin{proof}
Indeed we need to show that $H^+\subset \Sigma.$
Note first that since $S\ucal\subset E,$ we have that
\[
\rho(S\ucal)\qpz=\qpz.
\]
This together with Lemma~\ref{l;epimorphic}(2) and the fact that $U\subset\ucal$ implies 
$\rho(H^+)\qpz=\qpz.$
In view of~\eqref{eq:E-observable} we thus get that $H^+\subset\mathbb E.$ Hence,
\[
W_H^-(s)\subset W^-_E(s)\subset E'.
\]
Therefore, $H^+=\langle W_H^-(s),W_H^+(s)\rangle\subset \Sigma$ as was claimed.  
\end{proof}


\section{The arithmetic case}\label{sec:arithmetic}
In this section we provide a more detailed description of the groups $\bbf$ and $\Sigma$
in Theorem~\ref{algeb-measure} in the arithmetic setting. 

We begin by fixing some notation.
Let $K$ be a global function field of characteristic $p>0.$ 
Let $\bbg$ be a connected, simply connected, semisimple group defined over $K.$ 
Suppose $\tcal$ is a finite set of places of $K,$ and put 
\[
G=\prod_{{v}\in \tcal}G_v,
\]
where $G_{{v}}:=\bbg(K_{{v}})$ and $K_v$ is the completion of $K$ at $v$.
Denote by $\ocal_\tcal$ the ring of $\tcal$-integers in $K$
and let $\Gamma$ be a finite index subgroup of $\bbg(\ocal_\tcal)$. 

Put $\kt=\prod_{v\in\mathcal T} K_{v}.$ 
We use the notation in Theorem~\ref{algeb-measure},
in particular, $\lt\subset \kt$ and $\sfield_v=K_v$ for all $v\neq w.$ 

Let $\bbf\subset\rcal_{\kt/\lt}(\bbg)$ be as in Theorem~\ref{algeb-measure}(1) and (2),
hence 
\[
\bbf(\lt)\subset\rcal_{\kt/\lt}(\bbg)(\lt)=\bbg(\kt).
\]
Let $\bbm$ denote the connected component of the identity 
in the Zariski closure of $\bbf(\lt)\cap \Gamma$ in $\bbg.$
Then $\bbm$ is a $K$-subgroup of $\bbg,$ see~\cite[Lemma 11.2.4(ii)]{Sp}.
In particular, $\bbm$ is a $K_v$-subgroup of $\bbg$ for all $v\in\mathcal T$.

Recall from Theorem~\ref{algeb-measure}(2) and Lemma~\ref{lem:f-m}(2) that 
$\bbf(\lt)\subset\bbm$ and that $\bbm$ equals the connected component of the identity 
in the Zariski closure of $\bbf(\lt)$ in $\bbg.$

The following is the standing assumption in this section.
\begin{enumerate}
\item[($\bbm$)] $\bbm$ is a connected, 
simply connected, semisimple group defined over $K.$   
Moreover, either ${\rm char}(K)>3$ or if ${\rm char}(K)=2,3,$ 
then all of the absolutely almost simple factors of $\bbm$ are of type $A$.\end{enumerate}

Assuming ($\bbm$) above, our goal is to describe the structure of the group
\be\label{eq:def-D'}
D':=\overline{F^+(\la)\Bigl({\bbf}(\sfield_\tcal)\cap\Gamma\Bigr)},
\ee
where the closure is with respect to the Hausdorff topology.

Let $\{\bbm^*_i:1\leq i\leq r\}$ denote the $K$-almost simple factors of $\bbm.$   
Since $\bbm$ is simply connected, see ($\bbm$), we have 
\be\label{eq:M-K-factors}
\bbm=\prod_{i=1}^r\bbm^*_i.
\ee
Therefore, for all $1\leq i\leq r$, there exists a separable extension $\gfield_i/\gfield$ and 
a connected, simply connected, absolutely almost simple, $K_i$-group, 
$\check{\bbm}_i$, so that 
\[
{\bbm}_i^*=\mathcal R_{\gfield_i/\gfield}(\check{\bbm}_i).
\] 
In particular, $\check{\bbm}_i$ is naturally identified with $\prod_{j=1}^{b_i}{}^{\sigma_{i,j}}\check{\bbm}_i$
where $\{\sigma_{i,j}\}$ are different Galois embeddings of $\gfield_i$ 
into the separable closure of $\gfield$, see~\cite[Ch.~1, \S1.7]{Mar5}.

\begin{thm}[Special case]\label{thm:arithmetic-irred}
Assume that $\bbm$ is $K$-simple, i.e.\  $\bbm=\rcal_{K_1/K}(\check{\bbm}_1).$
Moreover, assume that $\tcal=\{w\}$. Then, there exist
\begin{enumerate}
\item an infinite subfield $K'\subset K_1$, 
\item a connected, simply connected, absolutely almost simple, $K'$-group, $\bbe'$, and
\item a $K_1$-isomorphism $f:\bbe'\times_{K'}K_1\to\check\bbm_1,$
\end{enumerate}
so that 
\[
D'\Gamma=f(\bbe'(A))\Gamma
\]
where $A$ is the closure of $K'$ in $K_1\otimes_{K}K_w.$
\end{thm}

Theorem~\ref{thm:arithmetic-irred} is a special case of the following more general statement.
To state this result we need some more notation. 

We will work with (commutative) semisimple rings, 
$\Upsilon=\oplus_{j}\Upsilon_j$, where $\Upsilon_j$ is a field for each $j.$ By a unital, semisimple, subring of 
$\Upsilon$ we always mean a subring with the same multiplicative identity element.

By a linear algebraic group, ${\mathbb B}$, over $\Upsilon=\oplus_{j}\Upsilon_j$
we mean $\bbb=\coprod_j{\mathbb B}_j$ where each 
${\mathbb B}_j$ is a linear algebraic group over $\Upsilon_j.$
The adjoint representation of $\bbb$ on ${\rm Lie}({\bbb})=\oplus_j{\rm Lie}(\bbb_j)$
is the direct sum of the adjoint representations of ${\bbb}_j$ on ${\rm Lie}(\bbb_j),$
and the group of $\Upsilon$-points of $\bbb$
is ${\bbb}(\Upsilon)=\prod_j\bbb_j(\Upsilon_j).$  
Similarly, other notions are defined fiberwise.

\begin{thm}[The general case]\label{thm:arithmetic-red}
Let the notation be as in~\eqref{eq:M-K-factors} and Theorem~\ref{algeb-measure}. 
Then, there exist 
\begin{enumerate}
\item a unital, semisimple, subring $\oplus_{\alpha=1}^{r'}K'_\alpha\subset\oplus_{i=1}^rK_i,$
\item a nonempty subset $\oldaleph\subset\{1,\ldots,r'\}$ and a subset $\mathsf J\subset\{1,\ldots,r\}$
so that $\oplus_{\oldaleph}K'_\alpha$ is a unital semisimple subring of $\oplus_{\mathsf J}K_j,$ 
\item a fiberwise connected, simply connected, absolutely almost simple, 
$\oplus_{\alpha\in\oldaleph}K_\alpha'$-group, $\bbe'=\coprod_{\alpha\in\oldaleph}\bbe_\alpha$, and
\item a $\oplus_{\mathsf J}K_j$-isomorphism $f:\bbe'\times_{\oplus_{\oldaleph}K_\alpha'}\oplus_{\mathsf J}K_j\to\coprod_{\mathsf J}\check\bbm_j$, 
\end{enumerate}
so that 
\[
D'\Gamma={f}(\bbe'(A))\Gamma
\]
where $A$ is the closure of $\oplus_{\oldaleph}K_\alpha'$ in $\oplus_{\mathsf J}\Bigl( K_j\otimes_{K}(\oplus_\tcal K_v)\Bigr)$.
\end{thm}

The proof of Theorem~\ref{thm:arithmetic-red} occupies the rest of this section.
Let us briefly outline the strategy.
First, we describe the structure of the group $\bbf$. 
This is done using the classification results in~\cite{CGP}; 
we recalled what is required (and tailored it for our application) in \S\ref{sec:psd-red}, 
see Theorem~\ref{thm:M-semisimple} and Lemma~\ref{lem:F-is-F}. 
In the next step, we use a result of Pink,~\cite{Pink}, 
to provide a global model, see Lemma~\ref{lem:global-structure}. In the third and final step, 
we use strong approximation theorem and the fact that $D'\cap\Gamma$ is a lattice in $D'$ in order 
to tie the results from Lemma~\ref{lem:F-is-F} and Lemma~\ref{lem:global-structure} together 
and finish the proof.
Let us now turn to the details of the argument.

Set $\bbf':=[\bbf,\bbf]$ to be the commutator subgroup of $\bbf.$

Since $\sfield_v=K_v$ for all $v\neq w$, the definition of $\bbm$ implies that 
\be\label{eq:F-v-neq-w}
\bbf_v=\bbf'_v=\bbm\text{ for all $v\neq w.$}
\ee

\begin{lem}\label{lem:F-is-F}
Put $\field:=K_w$ and $\sfield:=\sfield_w.$
The groups $\bbf_w,$ $\bbf'_w$, and $\bbm_w$ satisfy the conclusions of 
Lemma~\ref{lem:f-m}, Lemma~\ref{lem:F-pseudo-red}, and Theorem~\ref{thm:M-semisimple}. 
In particular, we have 
\begin{itemize}
\item[(a)] there is a subfield $\sfield\subset\sfield'\subset \field$ with $\field/\sfield'$ 
a separable extension, and $\sfield'/\sfield$ a purely inseparable extension, 

\item[(b)] there is some $m\geq 1$ and for all $1\leq i\leq m$, there is a field $\sfield\subset\sfield_i\subset \sfield'$, 
in particular, $\sfield_i/\sfield$ is a purely inseparable extension,

\item[(c)] for all $1\leq i\leq m,$ there is an $\sfield_i$-simple, connected, 
simply connected, $\sfield_i$-group ${\bbl}_i,$ and

\item[(d)] there is an isomorphism 
$
\iota: \prod_{i=1}^m\bbl_i\times_{\sfield_i}\sfield'\to\bbl,
$
where $\bbl$ is the irreducible component of the identity in the 
(fiberwise) Zariski closure of $\bbf(\lt)\cap \Gamma$ in $\mathcal R_{k/l}(\bbm)$,
\end{itemize}
so that the following hold.

\begin{enumerate} 
\item $\bbf_w'(\sfield)=\iota\Bigl(\prod_{i=1}^m{\bbl}_i(\sfield_i)\Bigr),$
 
\item $\bbf_w(\sfield)/\bbf_w'(\sfield)$ is a compact, abelian group. 
\end{enumerate}
\end{lem}

\begin{proof}
Recall that $\Gamma$ is diagonally embedded in $\bbg(\field_{\mathcal T}).$
By the definition, $\bbf$ is the fiberwise Zariski closure of a subgroup $\mathsf B\subset\Gamma.$  
Recall also from Theorem~\ref{algeb-measure}(2) that $\bbf$ is connected.
Moreover, $\bbm$ is defined to be the connected component of the identity 
in the Zariski closure of $\bbf(\lt)\cap \Gamma$ in $\bbg$
which satisfies the conditions in \S\ref{sec:psd-red} by our assumption ($\bbm$).
The claims thus follow.
\end{proof}

\begin{cor}\label{cor:f-m-gamma}
Let the notation be as in Lemma~\ref{lem:F-is-F}; we add the subscript $w$ to emphasize the place $w$. 
Let $\Delta'_w$ denote the projection of $\bbf'(l_\tcal)\cap\Gamma$ to $G_w$. Then 
$
\iota^{-1}(\Delta'_w)
$ 
is fiberwise Zariski dense in $\coprod_{i}{\bbl_{i,w}}.$ 
\end{cor}

\begin{proof}
Recall from Lemma~\ref{lem:F-is-F}(1) that $\bbf'(\sfield)=\iota\Bigl(\prod_i{\bbl}_i(\sfield_{i,w})\Bigr).$
Let ${\bbp}$ denote the (fiberwise) Zariski closure of 
$\iota^{-1}(\Delta_w')$ in $\coprod_{i}{\bbl_{i,w}}.$
Then $\Delta_w\subset\iota(\bbp)$.

Since $\bbl$ is semisimple and equals is the irreducible component of the identity in the 
(fiberwise) Zariski closure of $\bbf(\lt)\cap \Gamma$ in $\mathcal R_{K_w/l_w}(\bbm)$,
we get that $\Delta_w'$ is Zariski dense in $\bbl$.
The claim thus follows in view of the fact that $\iota$ is an isomorphism.
\end{proof}

Since $\bbm$ is simply connected, we can write $\bbm=\prod_{i=1}^{n_v}\bbm_{i,v}$
where $\bbm_{i,v}$ is $K_v$-almost simple for all $1\leq i\leq n_v$.  
Therefore, 
\[
\bbm_{i,v}=\rcal_{K_{i,v}/K_v}(\bbm_{i,v}^\dagger)
\] 
where $K_{i,v}/K_v$ is a finite separable extension and $\bbm_{i,v}^\dagger$ is an absolutely almost simple 
$K_{i,v}$-group, see~\cite[Ch.~1, \S1.7]{Mar5}, for all $1\leq i\leq n_v.$ 

Similarly let us write 

\begin{itemize}
\item $\bbl_w=\prod_{i=1}^{d_w}{\bbl}^*_{i,w}$, product of 
$\sfield'_w$-simple factors. Then ${\bbl}^*_{i,w}=\rcal_{\sfield'_{i,w}/\sfield'_w}(\check\bbl_{i,w})$
where $\sfield'_{i,w}/\sfield'_{w}$ is separable and $\check\bbl_{i,w}$ is an absolutely almost simple 
$\sfield'_{i,w}$-group.

\item For any $1\leq i\leq m=m_{w},$ we write $\bbl_{i,w}=\rcal_{\sfield^\dagger_{i,w}/\sfield_{i,w}}(\bbl^\dagger_{i,w})$
where $\sfield^\dagger_{i,w}/\sfield_{i,w}$ is separable and $\bbl^\dagger_{i,w}$ is an absolutely almost simple $\sfield^\dagger_{i,w}$-group.
\end{itemize}

For all $v\neq w,$ we have $\sfield_v=\sfield'_v=K_v$ and we put $\bbl_{i,v}={\bbl}_{i,v}^*=\bbm_{i,v}$, 
$\sfield_{i,v}^\dagger=\sfield'_i=K_{i,v},$ $\bbl_{i,v}^\dagger=\check\bbl_{i,v}=\bbm_{i,v}^\dagger,$ and $m_v=d_v=n_v.$

\begin{lem}\label{lem:L-Li-relation}
\begin{enumerate}
\item $\oplus_{j=1}^{d_v}\sfield'_{j,v}$ is a unital, semisimple, subring of $\oplus_{i=1}^{n_v}K_{i,v}$ and 
\[
\coprod_{j=1}^{d_v}\check\bbl_{j,v}\times_{\oplus_{j=1}^{d_v}\sfield'_{j,v}}\oplus_{i=1}^{n_v}K_{i,v}
\] 
is isomorphic to $\coprod_{i=1}^{n_v}\bbm_{i,v}^\dagger$.
\item There is a partition $\{1,\ldots,d_v\}=\mathcal J_{1,v}\cup\cdots\cup\mathcal J_{m_{v},v}$
so that $\iota=(\iota_i)$, and for every $1\leq i\leq m_{v}$ we have
\[
\iota_i:\rcal_{\sfield^\dagger_{i,v}/\sfield_{i,v}}(\bbl^\dagger_{i,v})\times_{\sfield_{i,v}}\sfield'_v\to\prod_{\mathcal J_i}\rcal_{\sfield'_{j,v}/\sfield_v'}(\check\bbl_{j,v})
\]
is an isomorphism, see Lemma~\ref{lem:F-is-F}.
In particular, for all $1\leq i\leq m_{v}$ and all $j\in\mathcal J_{i,v}$ we have 
\begin{enumerate}
\item $\sfield^\dagger_{i,v}\subset\sfield'_{j,v}$, 
\item the composite field $\sfield^\dagger_{i,v}\sfield'_v$ equals $\sfield'_{j,v},$ and 
\item $\bbl^\dagger_{i,v}\times_{\sfield_{i,v}}\sfield'_{j,v}$ is isomorphic to $\check\bbl_{j,v}$.
\end{enumerate}
\item The isomorphism $\iota$ gives rise to an embedding of $\oplus_{j=1}^{m_v} \sfield^\dagger_{j,v}$
into $\oplus_{i=1}^{n_v}\gfield_{i,v}$ as unital, semisimple, rings. More explicitly, this embedding is given
as follows. Let $B$ be 
the total ring of quotients of the ring generated by
\[
\Bigl\{\Bigl({\rm tr}(\rho_1(g)),\ldots,{\rm tr}(\rho_{n_v}(g))\Bigr):g\in{\bbf}_v(\sfield_v)\Bigr\},
\]
where $\rho_i$ is the unique non-trivial subquotient of the adjoint representation of $\Bigl({\bbm}_{i,v}^\dagger\Bigr)^{\rm ad}$ for each $1\leq i\leq n_v$. Then, $B=\oplus_{j=1}^{m_v}B_j$ and
\[
{\sfield^\dagger_{j,v}:=\begin{cases}B_j&\text{ if $\Bigl({\rm char}( \sfield^\dagger_{j,v}),\bbl_{j,v}^{\dagger}\Bigr)\neq (2,\SL_2)$}\\
\{c:c^2\in B_j\}&\text{ if $\Bigl({\rm char}( \sfield^\dagger_{j,v}),\bbl_{j,v}^{\dagger}\Bigr)= (2,\SL_2)$}\end{cases} }.
\]
\end{enumerate}
\end{lem}

\begin{proof}
The lemma in the case $v\neq w$ is clear, we thus assume $v=w$. 

Using the transitivity of the restriction of scalars functor, we have 
\[
\rcal_{K_w/\sfield_w'}(\bbm)=\prod_{i=1}^{n_w}\rcal_{K_{i,w}/\sfield'_{w}}(\bbm_{i,w}).
\]
Now by the definition of $\bbl_w$, see Lemma~\ref{lem:F-is-F}(d), we have the restriction to 
$\bbl_w\times_{\sfield'_w}K_w$ of the natural projection 
\[
q:\rcal_{K_w/\sfield_w'}(\bbm)\times_{\sfield'_w} K_w\to\bbm
\] 
is a surjection. Since $K_w/\sfield'_w$ is a separable extension part~(1) follows.

The first claim in part (2) is a direct corollary of Lemma~\ref{lem:F-is-F}(c) and the above definitions.
The other claims follow from the uniqueness of presentation as the restriction of scalars, 
see e.g.~\cite[Prop.\ A.5.14]{CGP}.   

To see part (3), first recall that $\bbf_w'(\sfield_w)$ is Zariski dense in $\bbm.$ 
The claim thus follows follows from~\cite[Ch.~I, Cor.\ 1.4.8]{Mar5} and Lemma~\ref{lem:F-is-F}(1).
\end{proof}

We now use the $K$-structure of $\bbm$ in order to provide a {\em global model}.

Recall from~\eqref{eq:M-K-factors} that $\{\bbm^*_i:1\leq i\leq r\}$ denotes the $K$-almost simple factors of $\bbm$ and  
\[
\bbm=\prod_{i=1}^r\bbm^*_i=\prod_{i=1}^r\mathcal R_{\gfield_i/\gfield}(\check{\bbm}_i).
\] 
Since $K\subset K_v$ for all $v,$ we have $\bbm_i^*$ is a semisimple, 
simply connected, $K_v$-group for all $1\leq i\leq r$ and all $v\in\tcal.$ 
Therefore, there exists a partition $\{1,\ldots,n_v\}=\mathcal I_{1,v}\cup\cdots\cup\mathcal I_{r,v}$ so that for 
all $1\leq j\leq r$ we have 
\be\label{eq:local-global}
\bbm_{j}^*\times_KK_v=\prod_{i\in\mathcal I_{j,v}}\rcal_{K_{i,v}/K_v}(\bbm_{i,v}^\dagger)=\prod_{v'|v}\rcal_{(K_j)_{v'}/K_v}(\check\bbm_j\times_{K_j}(K_j)_{v'})
\ee
where $(K_j)_{v'}$ is the completion of $K_j$ at $v'$. 
In particular, for all $i\in\mathcal I_{j,v}$ we have $\gfield_{i,v}=\gfield_{j}\gfield_{v}=(\gfield_{j})_{v'}$ for
some $v'|v$, and $\bbm_{i,v}^\dagger$ is isomorphic to $\check\bbm_j\times_{K_j}(\gfield_{j})_{v'}.$ 

Put
\be\label{eq:def-Lambda}
\Lambda:=\bbm(\field_\tcal)\cap\Gamma\subset\bbm(\ocal_\tcal).
\ee
Then, $\Lambda$ is a finite index subgroup of $\bbm(\ocal_\tcal)$, in particular, it is a lattice in $\bbm(\field_\tcal).$

Define $\Delta:={\bbf}(\sfield_\tcal)\cap\Lambda$ and 
$\Delta':={\bbf}'(\sfield_\tcal)\cap\Lambda.$
Let $\bar\Delta$ and $\bar\Delta'$ be the images of $\Delta$ and $\Delta'$ 
in $\bbm^{{\rm ad}}$, respectively. Then, $\bar\Delta'$ is Zariski dense in $\bbm^{{\rm ad}}.$

Let $\Delta_{w}$ be the projection of 
$\Delta$ to $\bbf_w(\sfield_w)$ and put  
\be\label{eq:delta-delta'}
\Delta'_{w}:=\text{the projection of $\Delta'$ to $\bbf'_w(\sfield_w)$}.
\ee
Similarly, define $\Delta_{v}$ and $\Delta'_v$ for all $v\in\tcal$.
Using Lemma~\ref{lem:F-is-F}(d), we let $\bar\Delta_v$ and $\bar\Delta'_v$ denote the images of
$\Delta_v$ and $\Delta'_v$ in the adjoint group $\prod_{j=1}^{m_v}\Bigl(\bbl_{j,v}^{\dagger}\Bigr)^{\rm ad}(\sfield_{j,v}^\dagger)$.

\begin{lem}\label{lem:global-structure}
There exists a unital, semisimple, subring 
\[
\oplus_{\alpha=1}^{r'}K'_\alpha\subset\oplus_{i=1}^rK_i,
\]
so that the following hold.
\begin{enumerate}
\item There exists a fiberwise absolutely almost simple, connected, simply connected,
$\oplus_{\alpha=1}^{r'}\gfield'_\alpha$-group $\bbe=\coprod_{\alpha=1}^{r'}{\mathbb E}_\alpha$, 
and an isomorphism
\[
\psi:{\bbe}^{{\rm ad}}\times_{\oplus_{\alpha=1}^{r'}\gfield'_\alpha}\oplus_{i=1}^r\gfield_i\to\coprod_{i=1}^r{\check\bbm_i}^{\rm ad}
\] 
so that $\bar\Delta'\subset \psi\Bigl({\bbe}^{{\rm ad}}(\oplus_{\alpha=1}^{r'}\gfield'_\alpha)\Bigr)$.

\item For every $v\in\tcal$, we have $\oplus_{\alpha=1}^{r'}\gfield'_\alpha\subset\oplus_{j=1}^{m_v}\sfield_{j,v}^\dagger\subset\oplus_{i=1}^{n_v}\gfield_{i,v}$.

\item For every $v\in\tcal$, there is an isomorphism 
\[
\phi_{v}: {\bbe}^{\rm ad}\times_{\oplus_{\alpha=1}^{r'}\gfield'_\alpha}\oplus_{j=1}^{m_v}\sfield_{j,v}^\dagger\to\coprod_{j=1}^{m_v}\Bigl({\bbl}_{j,v}^\dagger\Bigr)^{\rm ad}
\]
so that 
$
\iota^{-1}\Bigl(\bar{\Delta}'_{v}\Bigr)\subset \phi_{v}\Bigl({\bbe}^{\rm ad}(\oplus_{\alpha=1}^{r'}\gfield'_\alpha)\Bigr)
$
where $\iota$ is the isomorphism introduced in Lemma~\ref{lem:F-is-F}(d). 
\end{enumerate}
\end{lem}

\begin{proof}
Write $\bar\Delta'=\{\gamma_1,\gamma_2,\ldots\}$
and for each $s\geq 1$ define the subgroup 
\[
\bar\Delta'_s:=\langle\gamma_1,\ldots,\gamma_s\rangle.
\] 
Similar to the above discussion, define $\bar\Delta'_{s,v}$ for all $v\in\tcal$.
  
Recall that $\bar\Delta'$ is Zariski dense in ${\bbm}^{\rm ad},$ 
and that by Corollary~\ref{cor:f-m-gamma}, we have $\bar\Delta'_v$ is (fiberwise) Zariski dense in 
$\coprod_{j=1}^{m_v}\Bigl({\bbl}_{j,v}^\dagger\Bigr)^{\rm ad}.$
Therefore, there exists some $s_0$ so that for all $s\geq s_0$ we have $\bar\Delta_{s,v}'$ is Zariski dense 
in $\coprod_{j=1}^{m_v}{\Bigl(\bbl^\dagger_{j,v}\Bigr)^{\rm ad}}$, 
and also $\bar\Delta_s'$ is Zariski dense in ${\bbm}^{\rm ad}.$ 
Throughout the proof we always assume that $s\geq s_0.$

For all $s\geq s_0,$ the group $\bar\Delta'_s$ satisfies~\cite[Assump.~2.1]{Pink}. Therefore, 
by~\cite[Thm.~3.6]{Pink} we have the following. There exist
\begin{enumerate}
\item[(a)] a unital, semisimple, subring $\Upsilon_{s}\subset\oplus_{i=1}^r\gfield_i$,
\item[(b)] a fiberwise absolutely almost simple, connected, simply connected, $\Upsilon_{s}$-group $\bbe_s,$
\item[(c)] an isogeny $\phi_s:{\mathbb E}^{\rm ad}_{s}\times_{\Upsilon_{s}}\oplus_{i=1}^rK_i\to\coprod_{i=1}^r\check{\bbm}_i^{\rm ad}$ with nowhere vanishing derivative,
\end{enumerate} 
so that $\bar\Delta'_s\subset\phi_s({\mathbb E}^{\rm ad}_s(\Upsilon_{s})).$ 
Moreover, $\Upsilon_{s}$ is unique and $({\mathbb E}_s,\phi_s)$ is unique up to a 
unique $\Upsilon_{s}$-isomorphism.

We first note that in view of our assumption in small characteristics, see (${\bbm}$),
it follows from~\cite[Thm.~1.7(b)]{Pink} that the isogeny $\phi_s$ in (c) 
is an isomorphism.

Let $\Upsilon'_s$ denote the total ring of quotients of
\[
\Bigl\{\Bigl({\rm tr}(\rho_1(\gamma)),\ldots, {\rm tr}(\rho_r(\gamma))\Bigr):\gamma\in\bar\Delta'_s\Bigr\}.
\]
where $\rho_i$ is the unique non-trivial subrepresentation of the adjoint representation of 
$\check{\bbm}_i^{{\rm ad}}.$ Write $\Upsilon'_s=\oplus_{\alpha=1}^{r'_s} B_{s,\alpha}$.

Then by~\cite[Prop.~3.10]{Pink}, we have $\Upsilon_s=\oplus_{\alpha=1}^{r'_s} K'_{s,\alpha}$
where 
\be\label{eq:def-field-def-Fi}
{K'_{s,\alpha}:=\begin{cases}B_{s,\alpha}&\text{ if $\Bigl({\rm char}( K'_{s,\alpha}),\bbe_{s,\alpha}\Bigr)\neq (2,\SL_2)$}\\
\{c:c^2\in B_{s,\alpha}\}&\text{ if $\Bigl({\rm char}( K'_{s,\alpha}),\bbe_{s,\alpha}\Bigr)= (2,\SL_2)$}\end{cases} }.
\ee

We get from~\eqref{eq:def-field-def-Fi} and (a) that 
$\Upsilon_{s}\subset \Upsilon_{s+1}\subset\cdots\subset\oplus_{i=1}^r{\gfield_i}.$
Therefore, there is some $s_1\geq s_0$ so that for all $s\geq s_1$ we have $\Upsilon:=\Upsilon_{s}=\Upsilon_{s+1}$.
Let $s\geq s_1$ for the rest of the argument.

Let us write $\Upsilon=\oplus_{\alpha=1}^{r'}K'_\alpha.$
We claim that there exists a fiberwise absolutely almost simple, connected, simply connected, $\Upsilon$-group, ${\bbe}=\coprod_{\alpha=1}^{r'}\bbe_\alpha$, and an isomorphism 
\be\label{eq:def-psi}
\psi:{\bbe}^{{\rm ad}}\times_{\Upsilon}\oplus_{i=1}^r\gfield_i\to\coprod_{i=1}^r{\check\bbm_i}^{\rm ad}
\ee
so that $({\bbe}_s,\phi_s)$ is uniquely isomorphic to $({\bbe},\psi)$ for all $s\geq s_1.$
To see this, note that $(\Upsilon,{\bbe}_{s+1},\phi_{s+1})$ 
satisfies (a), (b), and (c) for $\bar\Delta'_s.$
Hence, $({\bbe}_{s},\phi_{s})$ is uniquely isomorphic to $({\bbe}_{s+1},\phi_{s+1})$ which implies the assertion with $(\bbe,\psi):=({\bbe}_{s_1},\phi_{s_1}).$

We now claim that
\be\label{eq:lambda-psi}
\bar\Delta'\subset\psi({\bbe}^{\rm ad}(\Upsilon))=\psi\Big(\textstyle\prod_{\alpha=1}^{r'}\bbe_\alpha(K'_{\alpha})\Bigr).
\ee 
To see this, note that
for all $s\geq s_1$ we have $\phi_s=\psi\circ f_s$ where $f_s$ is a unique 
$\Upsilon$-isomorphism between $\bbe_s^{\rm ad}$ and ${\bbe}^{\rm ad}.$ 
Also recall that $\bar\Delta'_s\subset\phi_s(\bbe^{\rm ad}_s(\Upsilon))$ for all $s.$
Since $f_s$ is an $\Upsilon$-isomorphism, we have
\[
f_s\Bigl({\bbe}^{\rm ad}_s(\Upsilon)\Bigr)={\bbe}^{\rm ad}(\Upsilon)
\] 
for all $s;$ we get that $\bar\Delta'_s\subset\psi\Bigl({\bbe}^{\rm ad}(\Upsilon)\Bigr)$ for all $s,$ and~\eqref{eq:lambda-psi} follows. Thus, we have shown that part~(1) holds for 
$\Upsilon=\oplus_{\alpha=1}^{r'}K'_\alpha$ and $(\bbe,\psi)$.

We now show that part (2) in the lemma also holds for $\Upsilon=\oplus_{\alpha=1}^{r'}K'_\alpha$.
By Lemma~\ref{lem:F-is-F}, for $v=w$, and by the definition, with $\iota={\rm id}$ otherwise, we have
\be\label{eq:Ldagger-L-F}
[\bbf_v,\bbf_v](\sfield)={\bbf}'_v(\sfield)=\iota\biggl(\prod_{j=1}^m\Bigl(\rcal_{\sfield^\dagger_{i,v}/\sfield_{i,v}}(\bbl^\dagger_{i,v})\Bigr)(\sfield_{i,v})\biggr)=\iota\biggl(\prod_{j=1}^m{\bbl}^\dagger_{j,v}(\sfield^\dagger_{j,v})\biggr)
\ee
Therefore, $\Upsilon'_s$ is contained in the total ring of quotients of the ring generated by 
\[
\Bigl\{\Bigl({\rm tr}(\rho_1(g)),\ldots,{\rm tr}(\rho_{n_v}(g))\Bigr):g\in{\bbf}'_{v}(\sfield)\Bigr\}
\]
in $\oplus_{i=1}^{n_v}\gfield_{i,v}$, where $\rho_i$ is the unique, non-trivial, subquotient of the adjoint representation of $({\bbm}_i^\dagger)^{\rm ad}$, see~\cite[Ch.~I, Cor.~1.4.8]{Mar5}. 
This and Lemma~\ref{lem:L-Li-relation}(3) 
imply part~(2) in view of~\eqref{eq:local-global} and~\eqref{eq:def-field-def-Fi}.

We now turn to the proof of (3). 
Note that in view of~\eqref{eq:local-global}, the isomorphism $\psi$ in~\eqref{eq:def-psi}
extends to an isomorphism 
\[
\psi_v:{\bbe}^{\rm ad}\times_{\Upsilon}\oplus_{i=1}^{n_v}\gfield_{i,v}\to 
\coprod_{i=1}^{n_v}\Bigl(\bbm_i^\dagger\Bigr)^{\rm ad},
\] 
for all $v\in\tcal$.

Let $\iota$ be as in Lemma~\ref{lem:F-is-F}(d), for $v=w$, and be the identity, otherwise. 
Then Lemma~\ref{lem:F-is-F}(d), Lemma~\ref{lem:L-Li-relation}, and the definitions imply that 
$\iota$ induces an isomorphism 
\[
\bar\iota_v:\coprod_{j=1}^{m_v}({\bbl_{j,v}^\dagger})^{\rm ad}\times_{\oplus_{j=1}^{m_v}\sfield^\dagger_{j,v}}\oplus_{i=1}^{n_v}\gfield_{i,v}\to\coprod_{i=1}^{n_v}\Bigl(\bbm^\dagger_i\Bigr)^{\rm ad}.
\]

We claim that 
\be\label{eq:def-phi}
\phi_{v}:=(\bar\iota_v)^{-1}\circ\psi_v
\ee 
satisfies part~(3) in the lemma.

First note that the above definitions imply that   
\[
\phi_{v}=(\bar\iota_v)^{-1}\circ\psi_v:\bbe^{\rm ad}\times_{\oplus_{\alpha=1}^{r'}K'_\alpha}\oplus_{i=1}^{n_v}\gfield_{i,v}\to\coprod_{j=1}^{m_v}({\bbl_{j,v}^\dagger})^{\rm ad}\times_{\oplus_{j=1}^{m_v}\sfield_{j,v}^\dagger}\oplus_{i=1}^{n_v}\gfield_{i,v}
\] 
is an isomorphism of $\oplus_{i=1}^{n_v}\gfield_{i,v}$-algebraic groups. 
Therefore, and in view of part~(2), part~(3) will follow if we show that $\phi_{v}$ is defined over $\oplus_{j=1}^{m_v}\sfield_{j,v}^\dagger.$ 

To see this, first note that 
\[
{\psi}_v^{-1}(\bar\Delta_v')\subset{\bbe}^{\rm ad}(\oplus_{\alpha=1}^{r'}K'_\alpha),
\] 
and ${\psi}_v^{-1}(\bar\Delta_v')$ is Zariski dense in $\bbe^{\rm ad}.$
The definitions, thus, imply that 
\begin{align}\label{eq:lambda-phi}
(\bar\iota_v)^{-1}(\bar\Delta_v')=\phi_{v}\circ\psi_v^{-1}(\bar\Delta_v')&\subset
\phi_{v}\Bigl(\bbe^{\rm ad}(\oplus_{\alpha=1}^{r'}K'_\alpha)\Bigr)\\
\notag&\subset\phi_{v}\Bigl(\bbe^{\rm ad}(\oplus_{j=1}^{m_v}\sfield_{j,v}^\dagger)\Bigr),
\end{align}
where we used part (2) in the last inclusion.

We now recall from Corollary~\ref{cor:f-m-gamma} that 
\[
(\bar\iota_v)^{-1}(\bar\Delta'_v)\subset\prod_{j=1}^{m_v}\Bigl({\bbl}_{j,v}^{\dagger}\Bigr)^{\rm ad}(\sfield_{j,v}^\dagger)
\]
is fiberwise Zariski dense in $\coprod_{j=1}^{m_v}({\bbl}_{j,v}^\dagger)^{\rm ad}$. 
Moreover, $\bbe^{\rm ad}(\oplus_{j=1}^{m_v}\sfield_{j,v}^\dagger)$ is 
fiberwise Zariski dense in $\bbe^{\rm ad}\times_\Upsilon\oplus_{j=1}^{m_v}\sfield_{j,v}^\dagger.$

Therefore,~\eqref{eq:lambda-phi} implies that $\phi_{v}$
is fiberwise defined over $\sfield^\dagger_{j,v}$, completing the proof of part (3) and the lemma.  
\end{proof}

By Lemma~\ref{lem:global-structure}, 
there is a partition $\{1,\ldots,m_v\}=J_{1,v}\cup\cdots\cup J_{r',v}$
so that $K'_\alpha$ is a unital, semisimple, subring of $\oplus_{j\in J_{\alpha,v}}\sfield_{j,v}^\dagger.$
In particular, we have $\Delta'=\prod_\alpha\Delta'_\alpha,$ $\psi=(\psi_\alpha)_\alpha$, 
$\phi_v=(\phi_{\alpha,v})_{\alpha}$, etc.

Abusing the notation, we let $\bullet=(\bullet_v)_{\tcal}$ for $\bullet=\iota,\psi,\phi$, etc.

\begin{cor}\label{cor:arithmetic-lattice}
\begin{enumerate}
\item The closure of $K'_\alpha$ in $\oplus_{v\in\tcal}\oplus_{j\in J_{\alpha,v}}\sfield^\dagger_{j,v}$
is identified with 
\[
\oplus_{u\in\tcal_\alpha}\oplus_{b=1}^{s_u}\tilde\field_{\alpha, b,u}
\]
where $\tcal_\alpha$ is a finite set of places in $K'_\alpha.$
\item Let $1\leq \alpha\leq r'$. 
Set $\Omega_\alpha':=\psi_\alpha^{-1}(\bar\Delta_\alpha')$, and put    
\[
\Omega_\alpha:=\pi^{-1}\biggl(\Omega'_\alpha\cap \pi\Bigl(\bbe_\alpha(K'_\alpha)\Bigr)\biggr),
\] 
where $\pi:\bbe_\alpha\to\bbe_\alpha^{\rm ad}$ is the covering map. 
Then, $\Omega_\alpha$ is an arithmetic lattice in $\prod_u\prod_b\bbe_\alpha(\tilde\field_{\alpha,b,u})$.  
\item $\Delta'$ is a finite index subgroup of $\Delta$.
\end{enumerate}
\end{cor}

\begin{proof}
Part~(1) is a consequence of the definitions.    

We now turn to the proof of part (2). 
First note that since $\phi_{\alpha,v}$ is an isomorphism, we have
\be\label{eq:delta'-E}
\phi_{\alpha,v}\biggl(\prod_{j\in J_\alpha}{\bbe}_\alpha^{\rm ad}(\sfield^\dagger_{j,v})\biggr)=\prod_{j\in J_\alpha}\Bigl(\bbl_{j,v}^\dagger\Bigr)^{\rm ad}(\sfield^\dagger_{j,v}).
\ee

In view of~\eqref{eq:delta-delta'} and Lemma~\ref{lem:F-is-F}(1), thus, we have
\[
\bar\Delta_\alpha'=\biggl(\prod_{\tcal}\prod_{j\in J_\alpha}\Bigl({\bbl}^{\dagger}_{j,v}\Bigr)^{\rm ad}(\sfield_{j,v}^\dagger)\biggr)\cap\Lambda
=\biggl(\prod_\tcal\phi_{\alpha,v}\Bigl(\prod_{j\in J_\alpha}{\bbe}_\alpha^{\rm ad}(\sfield^\dagger_{j,v})\Bigr)\biggr)\cap\Lambda,
\]
The above, part (1), and~\eqref{eq:lambda-psi} imply that  
\[
\bar\Delta'_\alpha=\biggl(\prod_{\tcal_\alpha}\phi_{\alpha,v}\Bigl(\prod_{b=1}^{s_u}{\bbe}_\alpha^{\rm ad}(\tilde\field_{\alpha, b,u})\Bigr)\biggr)\cap\Lambda,
\]
where we used the embedding in part (1).

Recall now that $\Lambda$ has finite index in ${\bbm}(\vintegers)$. 
These observations, together with the fact that 
$\psi_{\alpha}$ is induced from a $\oplus_{i=1}^r\gfield_i$-isomorphism, see~\eqref{eq:def-psi}, 
imply that $\Omega_\alpha'$ is commensurable with $\bbe_\alpha^{\rm ad}(\ocal_{\tcal_\alpha}).$
Part (2) thus follows, see e.g.~\cite[\S1]{BP-Finiteness}.

To see part (3) we argue as in the proof of Theorem~\ref{thm:M-semisimple}(2). Indeed, for every
$\gamma\in \Delta$, conjugation by $\gamma$ defines an automorphism of  $\prod_{i=1}^{r}\bbm(K_i)$ 
which preserves $\psi\Bigl(\prod_{\alpha=1}^{r'}\bbe_\alpha^{\rm ad}(\gfield'_\alpha)\Bigr).$
The claim thus follows from part (2) and the fact that arithmetic groups have finite index in their normalizer, 
see e.g.~\cite[\S1]{BP-Finiteness}.
\end{proof}

Let the notation be as in Lemma~\ref{lem:F-is-F} and Lemma~\ref{lem:global-structure}; 
let $\tilde\phi_v$, etc.\ denote the canonical lift of $\phi_v$, etc.\ to the simply connected covering group. 
Define  
\be\label{eq:L-E-l-pts}
L_\alpha^\dagger:=\prod_{v\in\tcal}\tilde\phi_{\alpha,v}^{-1}\Bigl(\prod_{j\in J_{\alpha,v}}\bbl_{j,v}^\dagger(\sfield_{j,v}^\dagger)\Bigr)=\prod_\tcal\prod_{J_{\alpha,v}}{\bbe}_{\alpha}(\sfield_{j,v}^\dagger),
\ee
and let $L^\dagger:=\prod_\alpha L_\alpha^\dagger$.

We also put 
\[
E_\alpha^\dagger:=\prod_{u\in\tcal_\alpha}\prod_{b=1}^{s_u}{\bbe}_{\alpha}(\tilde\field_{\alpha,b,u}),
\]
and write $E^\dagger:=\prod_\alpha E_\alpha^\dagger.$ 

Let $\la$ be the cocharacter of $\bbf_{w}$ which appears in Theorem~\ref{algeb-measure}. 
We get from the discussion in \S\ref{sec:psd-red}, see~\eqref{eq:F-standard-rep}, that $F^+(\la)\subset \bbf_w'(\sfield_w).$

Recall from~\eqref{eq:def-D'} that 
$
D'=\overline{F^+(\la)\Delta}.
$ 
Put $D'':=\overline{F^+(\la)\Delta'}\subset \bbf'(\lt)$ and define
\be\label{eq:def-D}
D:=\Bigl(\tilde{\phi}_{v}^{-1}\circ\iota_{v}^{-1}(D'')\Bigr)_{v\in\tcal}=\Bigl(\tilde{\psi}_{v}^{-1}(D'')\Bigr)_{v\in\tcal}\subset L^\dagger.
\ee

\begin{lem}\label{lem:Omega-D}
Let $\Omega:=\prod_{\alpha=1}^{r'}\Omega_\alpha$.  Then $\Omega\subset D$
is a lattice in $D.$
\end{lem}

\begin{proof}
First note that $\Omega\subset D$ follows from the definitions.
Now, by Theorem~\ref{algeb-measure}, we have $\Delta$ 
is a lattice in $D'$; therefore, $\Delta'$ is a lattice in $D''$ by Corollary~\ref{cor:arithmetic-lattice}(3). 
This, together with the fact that the above maps are topological isomorphisms, implies the claim. 
\end{proof}

In the rest of this section, we will use Corollary~\ref{cor:arithmetic-lattice} and Lemma~\ref{lem:Omega-D}
to finish the proof of Theorem~\ref{thm:arithmetic-red}.

Define
$
{\oldaleph}\subset\{1,\ldots,r'\}
$
to be the following subset. 
\be\label{eq:def-C}
\text{$\alpha\in\oldaleph$ if and only if the projection of $\la(\bbg_m)$ to $L^\dagger_\alpha$ is noncentral.}
\ee

Recall that ${\bbe}_\alpha$ is a connected, simply connected, absolutely almost simple, $K'_\alpha$-group 
for all $1\leq \alpha\leq r'.$ 
Then,~\cite[Ch.~1, Prop.\ 1.5.4 and Thm.\ 2.3.1]{Mar5}
and Lemma~\ref{lem:F-is-F}(1) imply that for every $\alpha\in{\oldaleph}$ there exists a (maximal) nonempty subset $J_\alpha(\la)\subset J_{\alpha,w}$ so that 
\be\label{eq:J-D-E}
\prod_{{\alpha\in\oldaleph}}\prod_{j\in J_\alpha(\la)}{\bbe}_{\alpha}(\sfield_{j,w}^\dagger)\subset D.
\ee

The following argument applies to any $\alpha\in{\oldaleph}$ and any $j\in J_\alpha(\la)$;
for the sake of an explicit exposition, however, we assume $1\in J_\alpha(\la)$ for some
$\alpha\in{\oldaleph}$ and work with this $\alpha$ and $j=1$.

Fix $Q,$ a compact open subgroup of $L^\dagger$. 
Since ${\bbe}_{\alpha}(\sfield_{1,w}^\dagger)\subset D$ is a normal subgroup of 
$L^\dagger$ and $Q$ is compact and open in $L^\dagger$, 
we get that ${\bbe}_{\alpha}(\sfield_{1,w}^\dagger)Q$ is an open and closed subgroup of $L^\dagger$. 

\begin{lem}\label{lem:R-lattice}
$\Bigl({\bbe}_{\alpha}(\sfield_{1,w}^\dagger)Q\Bigr)\cap\Omega$ is a lattice in 
${\bbe}_{\alpha}(\sfield_{1,w}^\dagger)Q.$
\end{lem}

\begin{proof}
Put $R:={\bbe}_{\alpha}(\sfield_{1,w}^\dagger)$ and let $Q':=Q\cap D$.
Then $Q'$ is a compact open subgroup of $D.$ Recall that ${RQ}$ is an open and closed subgroup of
$L^\dagger$. Put $R':= RQ\cap D$. 
Then $R\subset R'$ and $R'$ is an open and closed subgroup of $D.$ 
We first show that
$
\Sigma:=R'\cap \Omega
$
is a lattice in $R'.$

Recall from Lemma~\ref{lem:Omega-D} that $\Omega$ is a lattice in $D.$  
Since $R'$ is an open subgroup of $D$ we have $R'\Omega$ is an open and closed subset of $D.$
Therefore, the orbit map is a homeomorphism from $R'\Omega/\Omega$ 
onto its image in $D/\Omega$ equipped with the subspace topology. 
Now the restriction of the $D$-invariant measure on $D/\Omega$ to $R'\Omega/\Omega$ 
pulls back to the $R'$-invariant measure on $R'/\Sigma$. Hence, $\Sigma$ is a lattice in $R'$
as we claimed. 

Since $RQ$ and $R'$ are unimodular groups, $RQ/R'$
has an $RQ$-invariant measure. Moreover, $RQ/R'$ is compact, hence, 
$RQ/R'$ has a finite $RQ$-invariant measure.
We showed in the previous paragraph that $\Sigma$ is a lattice in $R'.$
Therefore, $\Sigma$ is a lattice in $RQ,$ see e.g.~\cite[Lemma 1.6]{Raghunathan}. 
The claim now follows since $\bigl(RQ\bigr)\cap\Omega$
is discrete and contains $\Sigma.$
\end{proof}

Let ${\rm p}:L^\dagger\to{\bbe}_{\alpha}(\sfield_{1,w}^\dagger)$ be the natural projection.

\begin{cor}\label{cor:delta-lattice}
${\rm p}\Bigl(\bigl({\bbe}_{\alpha}(\sfield_{1,w}^\dagger)Q\bigr)\cap\Omega\Bigr)$ is a lattice in ${\bbe}_{\alpha}(\sfield_{1,w}^\dagger)$.
\end{cor}

\begin{proof} 
The map ${\rm p}:{\bbe}_{\alpha}(\sfield_{1,w}^\dagger)Q\to {\bbe}_{\alpha}(\sfield_{1,w}^\dagger)$ 
is surjective and has compact kernel. 
Therefore, ${\rm p}$ is a surjective, proper map. This and Lemma~\ref{lem:R-lattice} 
imply that ${\rm p}\Bigl(\bigl({\bbe}_{\alpha}(\sfield_{1,w}^\dagger)Q\bigr)\cap\Omega\Bigr)$ is a lattice
in ${\bbe}_{\alpha}(\sfield_{1,w}^\dagger).$ 
\end{proof}

Recall from Corollary~\ref{cor:arithmetic-lattice}(1) that 
\[
\oplus_{\tcal_\alpha}\oplus_{b=1}^{s_u}\tilde\field_{\alpha, b,u}\subset\oplus_\tcal\oplus_{J_{\alpha,v}}\sfield^\dagger_{j,v}.
\] 
Therefore, for every $v\in\tcal$ and any $j\in J_{\alpha,v}$, there exist some $u(v)=u_\alpha(v)\in\tcal_\alpha$, 
some $1\leq b(j)=b_{\alpha,v}( j)\leq s_{u(v)},$ and a (continuous) embedding of fields 
\[
\tau_{\alpha,j,v}:\tilde\field_{\alpha,b(j),u(v)}\to\sfield_{j,v}^\dagger.
\]


\begin{lem}\label{lem:M-1}
Recall that we fixed one $\alpha\in\oldaleph$ and assumed that $1\in J_\alpha(\la)$.
Let us write $u_\alpha(w)=u_0$; without loss of generality we may assume $b_{\alpha,w}(1)=1.$ 
Suppose $v\in\tcal$ and $j\in J_{\alpha,v}$ are so that
$\Bigl(u_\alpha(v),b_{\alpha,v}(j)\Bigr)=(u_0,1)$. 
Then $v=w$ and $j=1.$ 
\end{lem}

\begin{proof}
Let $\scal\subset\tcal$ be the set of all $v$ so that the equation
\[
\Bigl(u_\alpha(v),b_{\alpha,v}(j)\Bigr)=(u_0,1)
\] 
has a solution if and only if $v\in\scal$.
For all $v\in\scal$, let $J'_{\alpha,v}\subset J_{\alpha,v}$ be the set of all $j$ 
so that $b_{\alpha,v}( j)=1$ if and only if $j\in J_{\alpha,v}'$.  
We need to show that $\scal=\{w\}$ and that $J'_{\alpha,w}=\{1\}.$

By Lemma~\ref{lem:R-lattice}, we have $\Bigl({\bbe}_{\alpha}(\sfield_{1,w}^\dagger)Q\Bigr)\cap\Omega$ 
is a lattice in ${\bbe}_{\alpha}(\sfield_{1,w}^\dagger)Q.$ 

Recall that $\Omega\subset E^\dagger$. Moreover, ${\bbe}_{\alpha}(\sfield_{1,w}^\dagger)$ is not compact
since $\la$ is non-central. Hence, 
$
\Bigl({\bbe}_{\alpha}(\sfield_{1,w}^\dagger)Q\Bigr)\cap E^\dagger
$ 
is unbounded. 

Therefore, there exists a sequence $\{c_n\}\subset\tilde\field_{\alpha,1,u_0}$ 
so that $\{\tau_{\alpha,1,w}(c_n)\}$ is unbounded, but, $\{\tau_{\alpha,j,v}(c_n)\}$
is bounded for all $w\neq v\in\scal$ and all $j\in J'_{\alpha,v}$, and for all $1\neq j\in J'_{\alpha,w}$
when $v=w$. 
This contradicts the fact that $\{\tau_{\alpha,j,v}\}$'s are continuous, and finishes the proof.
\end{proof}

\begin{lem}\label{lem:field-of-def}
With the notation as in Lemma~\ref{lem:M-1}, we have
$\tilde\field_{\alpha,1,u_0}=\sfield^\dagger_{1,w}.$ 
\end{lem}

\begin{proof}
Recall from Corollary~\ref{cor:arithmetic-lattice}(1) that 
\[
\oplus_{\tcal_\alpha}\oplus_{b=1}^{s_u}\tilde\field_{\alpha, b,u}\subset\oplus_\tcal\oplus_{J_{\alpha,v}}\sfield^\dagger_{j,v},
\]
and that $\oplus_\tcal\oplus_{J_{\alpha,v}}\sfield^\dagger_{j,v}$ 
is a finitely generated $\oplus_{\tcal_\alpha}\oplus_{b=1}^{s_u}\tilde\field_{\alpha, b,u}$-module.

The projection~$\rm p$ is fiberwise defined over $\sfield_{j,v}^\dagger.$ 
Put
\[
P:={\rm p}\Bigl(E^\dagger\Bigr).
\]
Then, by~\cite[A.2]{BZ} we get $P$ is a constructible subgroup of ${\bbe}_{\alpha}(\sfield_{1,w}^\dagger)$. 
Hence, $P$ is a closed subgroup of ${\bbe}_{\alpha}(\sfield_{1,w}^\dagger)$.

Recall that $\Omega\subset E^\dagger;$ therefore, 
\[
{\rm p}\Bigl(\bigl({\bbe}_{\alpha}(\sfield_{1,w}^\dagger)Q\bigr)\cap\Omega\Bigr)\subset P.
\]
Moreover, by Corollary~\ref{cor:delta-lattice}, we have 
${\rm p}\Bigl(\bigl({\bbe}_{\alpha}(\sfield_{1,w}^\dagger)Q\bigr)\cap\Omega\Bigr)$ is a lattice in 
${\bbe}_{\alpha}(\sfield_{1,w}^\dagger).$

All together, we get $P$ is a non-discrete, closed subgroup of ${\bbe}_{\alpha}(\sfield_{1,w}^\dagger)$ 
with finite covolume, see~\cite[Lemma 1.6]{Raghunathan}.
In view of this and the fact that ${\bbe}_{\alpha}$ is a connected, absolutely almost simple, simply connected,
$K'_\alpha$-group, we get from~\cite[Ch.~1, Thm.\ 2.3.1 and Ch.~2, Thm.\ 5.1]{Mar5} that
\be\label{eq:P-and-E}
P={\bbe}_{\alpha}(\sfield_{1,w}^\dagger).
\ee
Using the definition of ${\rm p}$, we thus get the following.
\begin{align*}
\ker({\rm p}){\bbe}_{\alpha}(\tilde\field_{\alpha,1,u_0})&=\ker({\rm p})E^\dagger&&\text{Lemma~\ref{lem:M-1}}\\
&=L^\dagger&&\text{\eqref{eq:P-and-E}}\\
&=\ker({\rm p}){\bbe}_{\alpha}(\sfield_{1,w}^\dagger).
\end{align*}
Therefore, ${\bbe}_{\alpha}(\tilde\field_{\alpha,1,u_0})={\bbe}_{\alpha}(\sfield_{1,w}^\dagger)$
which implies $\tilde\field_{\alpha,1,u_0}=\sfield_{1,w}^\dagger,$ see e.g.~\cite[Ch.~1, Prop.\ 2.5.5]{Mar5}.
\end{proof}

A similar argument as in Lemma~\ref{lem:M-1} and Lemma~\ref{lem:field-of-def} implies the following.

\begin{lem}\label{lem:sfield-I}
Let $\alpha\in{\oldaleph}$.
\begin{enumerate}
\item Let $j_0\in J_\alpha(\la).$ Then 
$\Bigl(u_\alpha(v),b_{\alpha,v}(j)\Bigr)=\Bigl(u_\alpha(w),b_{\alpha,w}(j_0)\Bigr)$ 
if and only if $v=w$ and $j=j_0.$
\item There is a subset 
$
I_{\alpha}(\lambda)\subset\{1,\ldots, s_{u_{\alpha}(w)}\}
$ 
and a bijection
\[
b_\alpha:J_\alpha(\la)\to I_{\alpha}(\lambda)
\] 
so that $\sfield^\dagger_{j,w}=\tilde\field_{\alpha,b_\alpha(j),u_\alpha(w)}$ 
for all $j\in J_\alpha(\la)$.
\end{enumerate}
\end{lem}

\begin{proof}[Proof of Theorem~\ref{thm:arithmetic-red}]
\begin{itemize}
\item Let $\oplus_{\alpha=1}^{r'}K'_\alpha$ and $\bbe=\coprod_{\alpha=1}^{r'}\bbe_\alpha$ be as in Lemma~\ref{lem:global-structure}.
\item Let $\oldaleph$ be as in~\eqref{eq:def-C}, and define $\mathsf J$ as in Theorem~\ref{thm:arithmetic-red}(2) 
for this $\oldaleph$. 
\item Let $\bbe'=\coprod_{\alpha\in\oldaleph}\bbe_\alpha$. 
\item Let $A=\oplus_{\alpha\in\oldaleph}A_\alpha$ where $A_\alpha=\oplus_{\tcal_\alpha}\oplus_{b=1}^{s_u}\tilde\field_{\alpha, b,u}$ is as in Corollary~\ref{cor:arithmetic-lattice}(1).
\end{itemize}
Then, $\psi$ in Lemma~\ref{lem:global-structure}(2) induces a $\oplus_{\mathsf J}K_i$-isomorphism, $f$, between 
$\bbe'\times_{\oplus_{\oldaleph}K_\alpha'}\oplus_{\mathsf J}K_i$ and $\coprod_{\mathsf J}\check\bbm_i$. 
Moreover, $A$ is the closure of $\oplus_{\oldaleph}K_\alpha'$ in $\oplus_{i\in\mathsf J}\Bigl( K_i\otimes_{K}(\oplus_\tcal K_v)\Bigr)$.

By Lemma~\ref{lem:sfield-I}, for all $\alpha\in\oldaleph$, we have
\[
\bbe_\alpha(\sfield^\dagger_{j,w})=\bbe_\alpha(\tilde\field_{\alpha,b_\alpha(j),u_\alpha(w)}).
\]
Therefore, in view of the strong approximation theorem~\cite[Ch.~2, Thm.\ 6.7]{Mar5}
and Corollary~\ref{cor:arithmetic-lattice}
we get that
\[
D=\displaystyle\prod_{\alpha\in {\oldaleph}}E^\dagger_\alpha\prod_{\alpha\not\in{\oldaleph}}\Omega_\alpha,
\]
see~\eqref{eq:J-D-E}. This and~\eqref{eq:def-D} finish the proof. 
\end{proof}



\end{document}